\documentclass[final,hidelinks,onefignum,onetabnum]{siamart251216}

\usepackage[T1]{fontenc}
\usepackage[utf8]{inputenc}
\usepackage{amssymb,mathtools,bm}
\usepackage{algpseudocode}
\usepackage{booktabs}
\usepackage{enumitem}
\usepackage{graphicx}
\usepackage{flafter}
\usepackage[section]{placeins}
\usepackage{subcaption}
\usepackage{tabularx}
\usepackage{microtype}
\usepackage{xurl}
\newcolumntype{L}[1]{>{\raggedright\arraybackslash}p{#1}}
\newcolumntype{Y}{>{\raggedright\arraybackslash}X}

\setlist[itemize]{leftmargin=1.8em}
\setlist[enumerate]{leftmargin=1.8em}

\newsiamremark{assumption}{Assumption}
\newsiamremark{remark}{Remark}
\crefname{assumption}{Assumption}{Assumptions}
\crefname{remark}{Remark}{Remarks}

\newcommand{\RR}{\mathbb{R}}
\newcommand{\EE}{\mathbb{E}}
\newcommand{\ip}[2]{\langle #1,#2\rangle}
\newcommand{\argmin}{\mathop{\mathrm{argmin}}}
\newcommand{\dist}{\mathop{\mathrm{dist}}}
\newcommand{\dom}{\mathop{\mathrm{dom}}}
\newcommand{\range}{\mathcal{R}}
\newcommand{\supp}{\mathop{\mathrm{supp}}}
\newcommand{\KL}{\mathop{\mathrm{KL}}}
\newcommand{\relu}{\mathop{\mathrm{ReLU}}}
\newcommand{\softmax}{\mathop{\mathrm{softmax}}}

\newcommand{\ind}{\iota}
\newcommand{\cF}{\mathcal{F}}
\newcommand{\cT}{\mathcal{T}}
\newcommand{\cX}{\mathcal{X}}

\newcommand{\one}{\mathbf{1}}
\newcommand{\norm}[1]{\left\|#1\right\|}
\newcommand{\relint}{\operatorname{relint}}

\usepackage{xcolor}
\newcommand{\zy}[1]{\textcolor{black}{#1}}
\newcommand{\vl}[1]{\textcolor{black}{#1}} 
\newcommand{\hsiupdate}[1]{\textcolor{black}{#1}}

\headers{Randomized Row-Action Methods}{V. Leplat, Z. Dong, and J. Yin}

\title{A Unified Framework for Randomized Row-Action Methods
for Structured Data Fitting and Constraint Geometry}

\author{
Valentin Leplat\thanks{
Department of Mathematics and Operational Research, Universit\'e de Mons, Rue de Houdain 9, 7000 Mons, Belgium
(\email{Valentin.LEPLAT@umons.ac.be}).
}
\and
Zeyu Dong\thanks{Shanghai Research Institute for Intelligent Autonomous Systems, Tongji University,
Shanghai, 200092, China
  (\email{dongzeyu@tongji.edu.cn}).}
\and
Junfeng Yin\thanks{Key Laboratory of Intelligent Computing and Applications (Ministry of Education), School of Mathematical Sciences,
Tongji University,
Shanghai, 200092, China
  (\email{yinjf@tongji.edu.cn}).}
}

\ifpdf
\hypersetup{
  pdftitle={A Unified Framework for Randomized Row-Action Methods for Structured Data Fitting and Constraint Geometry},
  pdfauthor={Valentin Leplat, Zeyu Dong, and Junfeng Yin}
}
\fi

\begin{document}
\maketitle

\begin{abstract}
\vl{We study randomized row-action methods for problems with linear observations, convex constraints, and non-Euclidean geometry. Our starting point is a unified formulation in which the data-fitting term acts in observation space, while the feasible set and the regularizer encode structure in variable space. It covers constrained least squares, nonnegative and simplex-constrained models, regularized data fitting, and equality-constrained Bregman selection problems. The structure of these problems leads naturally to two algorithmic branches. The first uses primal Bregman-proximal row steps for soft data-fitting problems. The second uses dual Bregman--Kaczmarz steps for problems with hard equality constraints.} We first derive a common one-step inequality. We then give deterministic and stochastic convergence results for the soft branch under relative smoothness. For constrained least squares, the normalized row oracle has exact first- and second-moment identities. They yield an expected ergodic bound in the inconsistent case, an $O(1/K)$ bound in the interpolating case, and linear convergence of the actual iterates under a residual error bound. For equality-constrained selectors, we prove linear convergence in expected Bregman distance under a residual-to-Bregman error bound. We recall global sufficient conditions for this bound and derive local constants for the nonnegative quadratic geometry and the entropy geometry on the simplex. 
\vl{Numerical experiments illustrate the noisy stepsize trade-off and the local entropy selector theory, and show that entropic row-action updates can recover interpretable factors in a simplex-structured nonnegative matrix factorization problem.}
\end{abstract}



\section{Introduction}
Many inverse problems and data-fitting models have the form
\begin{equation}
\label{eq:intro-model}
    \min_{x\in C} D(Ax,y)+r(x),
\end{equation}
where $A\in\RR^{m\times n}$ is a linear observation operator. The set $C$ can impose nonnegativity, box constraints, or simplex constraints. The function $r$ can promote sparsity or another preferred structure. When $m$ is large, a full iteration may be expensive because it uses all rows of $A$. Row-action methods reduce this cost. At each iteration, they use one row or one small block of rows.

The Kaczmarz method is the classical example. For a consistent system $Ax=b$, it projects the current iterate onto one row hyperplane. Randomized sampling gives a linear expected rate in the consistent case \cite{Gordon1970ART,StrohmerVershynin2009}. Many extensions have been proposed for noisy systems, blocks, acceleration, constraints, and alternative sampling rules; see, for example, \cite{Needell2010Noisy,NeedellTropp2014,GowerRichtarik2015,ZouziasFreris2013,LeventhalLewis2010,TondjiNecoaraLorenz2024}. These methods are attractive because a single step is cheap and can use streaming access to the rows of $A$.

Structural constraints create a second difficulty. A row projection onto $a_i^\top x=b_i$ does not preserve nonnegativity or simplex feasibility. Projected row-action methods address this issue directly in the primal space. They combine a local row correction with a projection onto a convex set \cite{Popa2008CKE,WuXiang2020PRK}. Bregman methods take another route. They change the geometry of the update or work with a dual variable and a mirror map. This produces sparse, nonnegative, or entropy-based iterates depending on the chosen potential \cite{Bregman1967,schopfer2019linear,Schpfer2022ExtendedRK,TondjiNecoaraLorenz2024,DongXiaoYinYin2026,Dong2025SHBK}.

\vl{These lines of work are closely related, but they address two different targets that should be distinguished.} In a soft data-fitting model, the solution is a constrained minimizer and the optimal residual may be nonzero. A sampled row gives a local stochastic model of the objective. It is not a hard equation that should be satisfied exactly. In an equality-constrained selector model, the equations are imposed exactly and a strongly convex function selects one feasible point. A row update can then be viewed as a Bregman projection or as a coordinate step on the Fenchel dual.

The distinction matters in the inconsistent case. Suppose that
\[
    x^\star\in\argmin_{x\in C}\frac12\norm{Ax-y}_2^2
\]
has a nonzero residual. A feasibility theorem for $Ax=y$ cannot describe convergence to $x^\star$. Conversely, a stochastic projected-gradient argument does not use the dual selector structure available when the equations are exact. The algorithms share a row-action principle, but they need different assumptions and different proofs.

This paper makes this split explicit. We use \eqref{eq:intro-model} as a modeling umbrella and show in detail how several existing models fit it. We then develop two general algorithmic branches. The soft branch uses a primal Bregman-proximal row step. The hard branch uses a dual Bregman--Kaczmarz step followed by a geometry-dependent mirror map. The two branches coincide in some Euclidean consistent cases, but they are not interchangeable in general.

\vl{The unified formulation gives a common modeling framework, while the appropriate row-action update depends on the structure of the problem. The primal branch applies naturally when the data term decomposes by rows or blocks and the corresponding Bregman-proximal subproblem can be solved efficiently. Hard equality models with inconsistent observations can first use a range-correction step and then apply the selector update. In this way, the framework makes the common structure explicit and shows which row-action realization is suited to each model.}

A main motivation comes from low-rank matrix and tensor approximation. Block-coordinate methods repeatedly solve constrained linear subproblems. For example, fixing one factor in nonnegative matrix factorization gives many nonnegative least-squares problems. Simplex constraints also arise when the factors represent abundances or mixture weights. Row-action updates may be useful when these subproblems are large, sparse, or accessed in a streaming way. The same idea can later be applied to nonlinear factorizations in which an entrywise map is placed on top of a low-rank product.

\paragraph{Contribution and outline}
Our contributions are as follows.
\begin{enumerate}
    \item \vl{We introduce a unified formulation for structured row-action problems that covers constrained least squares, regularized fitting, nonnegative and simplex models, equality selectors, and range-corrected selector models.}
    \item \vl{Within this formulation, we distinguish two natural algorithmic branches. The soft branch uses primal Bregman-proximal row steps and allows a nonzero optimal residual. The hard branch uses dual Bregman--Kaczmarz updates and exact linear equations.}
    \item We establish a convergence hierarchy for the soft branch. A common Bregman-proximal inequality leads to deterministic and stochastic results under relative smoothness. For constrained least squares, exact oracle identities give sharper bounds.
    \item In the interpolating constrained least-squares case, we prove linear convergence of the actual projected iterates under a residual error bound. The classical randomized Kaczmarz rate is recovered as a special case.
    \item For equality-constrained selectors, we prove an expected linear contraction under a residual-to-Bregman error bound. We recall global sufficient conditions and give local constants for the nonnegative quadratic and entropy simplex geometries.
    \item \vl{We complement the theory with three numerical experiments on noisy NNLS, an entropy selector on the simplex, and a simplex-structured NMF problem. The experiments illustrate the predicted behavior and show how the proposed row-action updates can be used in a structured factorization model.}
\end{enumerate}
Section~\ref{sec:related} reviews the related work and positions the contribution. Section~\ref{sec:prelim} gives the notation and the basic Bregman tools. Section~\ref{sec:framework} presents the master formulation and the two algorithms. Section~\ref{sec:theory} gives the convergence analysis. Section~\ref{sec:numerics} reports the numerical experiments. Section~\ref{sec:conclusion} concludes the paper. Technical proofs and additional implementation details are collected in the appendices.

\section{Related work and positioning}
\label{sec:related}
This section reviews the three lines of work that are closest to the present paper. We then explain the role of the unified formulation.

\subsection{Classical and randomized Kaczmarz methods}
The Kaczmarz method was introduced as an iterative projection method for linear systems \cite{Kaczmarz1937} and later became a standard tool in algebraic reconstruction \cite{Gordon1970ART}. Strohmer and Vershynin showed that norm-proportional randomized sampling gives an expected linear rate for consistent systems \cite{StrohmerVershynin2009}. The rate depends on the scaled condition number $\norm{A}_F^2/(\sigma_{\min}^+(A))^2$.

The inconsistent case led to several different approaches. Needell showed that the basic randomized method converges to a noise-dependent neighborhood \cite{Needell2010Noisy}. Randomized extended Kaczmarz introduces an auxiliary correction in observation space and converges to a least-squares solution \cite{ZouziasFreris2013}. The sketch-and-project viewpoint places many randomized linear solvers in one Euclidean framework and clarifies the role of the sampling distribution \cite{GowerRichtarik2015}. Block methods reduce communication and can exploit groups of compatible rows \cite{NeedellTropp2014}. More recent work studies acceleration, momentum, adaptive sampling, and beyond-Krylov behavior \cite{Morshed2019ASKM,AldermanLuikartMarshall2024,DerezinskiNeedellRebrovaYang2025}.

These results mainly concern unconstrained linear systems or least-squares problems. The structure of the unknown is either absent or handled by a separate device.

\subsection{Projected and constrained row-action methods}
When $x$ must belong to a closed convex set, a natural update is a row correction followed by a projection. This idea appears in constrained ART and projected Kaczmarz methods \cite{Popa2008CKE,WuXiang2020PRK}. The consistent case is particularly favorable because the solution belongs to every sampled hyperplane. One can then combine the exact row identity with nonexpansiveness of the projection.

The noisy constrained case is more delicate. The minimizer of
\[
    \min_{x\in C}\frac12\norm{Ax-y}_2^2
\]
need not satisfy any sampled row equation. A fixed Kaczmarz step therefore keeps a nonvanishing stochastic perturbation at the optimum. This is the source of the stepsize-dependent neighborhood seen in noisy randomized row methods. To converge to the constrained minimizer, one needs a stochastic-optimization argument, a decreasing stepsize, variance reduction, or a correction mechanism adapted to the constraint.

The present paper separates this soft problem from exact feasibility. Our least-squares result is stated directly for the constrained minimizer. It does not assume that the optimal residual vanishes.

\subsection{Bregman--Kaczmarz and structural selectors}
Bregman projections replace Euclidean distance by a convex geometry \cite{Bregman1967}. In row-action form, one may update a dual vector and recover the primal variable through $x=\nabla f^*(u)$. With a quadratic $f$, this is the classical Kaczmarz method. With
\[
    f(x)=\lambda\norm{x}_1+\frac12\norm{x}_2^2,
\]
the terminal map is soft-thresholding and the method selects a sparse feasible point. Linear convergence of this randomized sparse Kaczmarz method was established in \cite{schopfer2019linear}. Extended and accelerated variants were later developed for inconsistent data and block updates \cite{Schpfer2022ExtendedRK,TondjiNecoaraLorenz2024,DongXiaoYinYin2026}.

This branch is naturally written as
\[
    \min_x f(x)\qquad\text{subject to}\qquad Ax=b.
\]
Strong convexity gives a unique selected minimizer even when the affine system has many feasible points. The proof is controlled by a residual-to-Bregman error bound. Existing global results use regularity properties of the subdifferential. For extended-valued geometries such as the nonnegative quadratic potential or the entropy on the simplex, local constants are often easier to obtain than a global spectral formula. We make this distinction explicit.

\subsection{Positioning of the present work}
\vl{The unified formulation isolates three ingredients shared by several row-action methods: a linear observation map, a data term in observation space, and a geometry or constraint in variable space. The target problem then determines the row-action update and the convergence argument.}

At the algorithmic level, the soft and hard branches use different updates. At the theoretical level, they use different target properties. The soft branch studies objective gaps to a constrained minimizer. The hard branch studies feasibility residuals and Bregman distance to a selected feasible point. This split avoids importing an exact-feasibility theorem into a noisy optimization problem.

\vl{The resulting framework is useful in two ways. First, it places projected row descent, mirror row descent, and Bregman selectors in one model. Second, it makes the conditions for an efficient row-action realization explicit. The data term should admit local row or block access, the corresponding Bregman-proximal map should be tractable, and structural selectors require suitable residual-to-Bregman control. These conditions also show where additional algorithmic ingredients are needed.}

\section{Preliminaries}
\label{sec:prelim}
We collect the basic definitions used in the two branches.

\subsection{Convex constraints and optimality}
Let $C\subseteq\RR^n$ be nonempty, closed, and convex. Its indicator is
\[
    \ind_C(x)=
    \begin{cases}
    0, & x\in C,\\
    +\infty, & x\notin C.
    \end{cases}
\]
The normal cone at $x\in C$ is
\[
    N_C(x)=\{v\in\RR^n:\ip{v}{z-x}\le0\ \text{for all }z\in C\}.
\]
For a differentiable convex function $F$, a point $x^\star\in C$ solves $\min_{x\in C}F(x)$ if and only if
\[
    0\in\nabla F(x^\star)+N_C(x^\star).
\]
The Euclidean projection is denoted by $P_C$. It is nonexpansive:
\[
    \norm{P_C(x)-P_C(z)}_2\le\norm{x-z}_2.
\]

\subsection{Bregman distances and mirror maps}
Let $h$ be differentiable on the relative interior of its domain. Its Bregman distance is
\begin{equation}
\label{eq:bregman-def}
    D_h(x,z)=h(x)-h(z)-\ip{\nabla h(z)}{x-z}.
\end{equation}
The distance is generally not symmetric. If $h$ is $\mu_h$-strongly convex with respect to a norm $\norm{\cdot}$, then
\begin{equation}
\label{eq:strong-bregman}
    D_h(x,z)\ge\frac{\mu_h}{2}\norm{x-z}^2.
\end{equation}
The three-point identity is
\begin{equation}
\label{eq:three-point}
    \ip{\nabla h(z)-\nabla h(y)}{y-x}
    =D_h(x,z)-D_h(x,y)-D_h(y,z).
\end{equation}

For a proper closed convex function $f$, the Fenchel conjugate is
\[
    f^*(u)=\sup_x\{\ip{u}{x}-f(x)\}.
\]
If $f$ is $\mu_f$-strongly convex, then $f^*$ is differentiable and $\nabla f^*$ is $1/\mu_f$-Lipschitz. The map $u\mapsto\nabla f^*(u)$ is the mirror map used in the hard branch.

For a possibly nonsmooth $f$, $u\in\partial f(x)$, and a reference point $z$, we use the generalized Bregman distance
\begin{equation}
\label{eq:generalized-bregman-prelim}
    D_f^u(x,z)=f(z)-f(x)-\ip{u}{z-x}.
\end{equation}
Fenchel equality gives the equivalent form \textcolor{black}{$D_f^u(x,z)=f^*(u)-\ip{u}{z}+f(z)$}.

Two geometries will be used repeatedly. For \textcolor{black}{$h(x)=\tfrac12\norm{x}_2^2$},
we have $D_h(x,z)=\frac12\norm{x-z}_2^2$. On the simplex \textcolor{black}{$\Delta_n=\{x\in\RR_+^n:\one^\top x=1\}$},
the negative entropy \textcolor{black}{$h(x)=\sum_{j=1}^n x_j\log x_j+\ind_{\Delta_n}(x)$}
induces the Kullback--Leibler distance and the mirror map $\softmax$.

On the simplex, the negative entropy is $1$-strongly convex with
respect to the Euclidean norm. Moreover, for $x,\hat x\in\relint(\Delta_n)$,
the generalized Bregman distance used in the hard branch satisfies \textcolor{black}{\mbox{$D_f^u(x,\hat x)=\KL(\hat x\Vert x)$}}.

\subsection{Rows, blocks, and sampling}
The rows of $A$ are denoted by $a_i^\top$, $i=1,\ldots,m$. We assume that zero rows are removed. The Frobenius norm is \textcolor{black}{\mbox{$\norm{A}_F^2=\sum_{i=1}^m\norm{a_i}_2^2$}}.
Norm-proportional sampling uses
\begin{equation}
\label{eq:norm-sampling-prelim}
    p_i=\frac{\norm{a_i}_2^2}{\norm{A}_F^2}.
\end{equation}
The history up to iteration $k$ is denoted by $\cF_k$. All conditional expectations are taken with respect to this filtration. The row formulas extend to blocks by replacing $a_i^\top$ with a submatrix and the scalar row model with a local block model.

\section{Unified formulation and two algorithmic branches}
\label{sec:framework}
This section has four goals. 
\vl{We introduce the master formulation, describe the models it covers, state when efficient row-action updates are available, and present the two algorithmic branches.}

\subsection{The master formulation}
We consider
\begin{equation}
\label{eq:master}
    \min_{x\in C} D(Ax,y)+r(x).
\end{equation}
Here $C\subseteq\RR^n$ is nonempty, closed, and convex. The function $D(\cdot,y)$ measures the fit in observation space, while $r$ acts directly on the unknown. Equivalently,
\begin{equation}
\label{eq:master-composite}
    \min_{x\in\RR^n} D(Ax,y)+R(x),
    \qquad R(x):=r(x)+\ind_C(x).
\end{equation}
This second form is convenient for proximal and Bregman-proximal updates.

A row-action method needs additional structure. In the simplest case,
\begin{equation}
\label{eq:row-separable}
    D(Ax,y)=\sum_{i=1}^m d_i(a_i^\top x,y_i).
\end{equation}
More generally, let $I_1,\ldots,I_s$ be a partition of the row indices. A block decomposition has the form
\[
    D(Ax,y)=\sum_{\ell=1}^s D_{I_\ell}(A_{I_\ell,:}x,y_{I_\ell}).
\]
\vl{For the row-action algorithms below, we assume such a row or block decomposition.}

\subsection{Representative special cases}
Table~\ref{tab:special-cases} summarizes the main cases. The last column is important. It indicates the natural algorithmic branch, not only the modeling reduction.

\begin{table}[t]
\centering
\caption{Representative instances of the master formulation \eqref{eq:master}.}
\label{tab:special-cases}
\small
\begin{tabularx}{\textwidth}{@{}L{0.22\textwidth}L{0.26\textwidth}L{0.24\textwidth}Y@{}}
\toprule
Model & Data term & Variable-space structure & Natural branch \\
\midrule
Constrained least squares & $\frac12\norm{Ax-y}_2^2$ & $x\in C$ & Soft primal branch \\
Nonnegative least squares & $\frac12\norm{Ax-y}_2^2$ & $C=\RR_+^n$ & Soft primal branch \\
Simplex least squares & $\frac12\norm{Ax-y}_2^2$ & $C=\Delta_n$ & Soft primal branch (Euclidean or entropic geometry) \\
Regularized data fitting & row-decomposable $D$ & $r\neq0$, $x\in C$ & Soft primal branch \\
Equality selector & $\ind_{\{y\}}(Ax)$ & strongly convex $f$ & Hard equality branch \\
Sparse selector & $\ind_{\{y\}}(Ax)$ & $f=\lambda\norm{x}_1+\frac12\norm{x}_2^2$ & Hard equality branch \\
Nonnegative selector & $\ind_{\{y\}}(Ax)$ & $f=\ind_{\RR_+^n}+\frac12\norm{x}_2^2$ & Hard equality branch \\
Entropy selector & $\ind_{\{y\}}(Ax)$ & entropy on $\Delta_n$ & Hard equality branch \\
Range-corrected selector & auxiliary fit in $\range(A)$ & strongly convex selector & Extended hard branch \\
\bottomrule
\end{tabularx}
\end{table}

\paragraph{Constrained and regularized data fitting}
With
\[
    D(Ax,y)=\frac12\norm{Ax-y}_2^2,
\]
problem \eqref{eq:master} gives constrained least squares. The choices $C=\RR_+^n$ and $C=\Delta_n$ give NNLS and simplex-constrained least squares. A nonzero $r$ adds sparsity, smoothness, or another regularizer. These models remain meaningful when the system is inconsistent. The target is the constrained minimizer, not a feasible point of $Ax=y$.

The same structure appears in block-coordinate methods. Fixing all but one factor in a matrix or tensor approximation often gives a family of constrained linear problems. Depending on the factor, the feasible set can be the nonnegative orthant, a simplex, a box, or an affine slice.

\paragraph{Equality-constrained selectors}
Taking $D(u,y)=\ind_{\{y\}}(u)$ gives
\begin{equation}
\label{eq:selector}
    \min_x f(x)\qquad\text{subject to}\qquad Ax=y,
\end{equation}
where $f$ includes both the selector and the structural constraint. If the system is feasible and $f$ is strongly convex, the selected minimizer is unique even when the affine system has many solutions.
Different choices of $f$ lead to the identity, soft-thresholding, ReLU, or softmax terminal maps.

\paragraph{Range-corrected selectors}
Extended Bregman--Kaczmarz methods can be written as
\begin{equation}
\label{eq:range-corrected}
    \min_x f(x)\quad\text{subject to}\quad Ax=\hat y,
    \qquad
    \hat y\in\argmin_{z\in\range(A)}g^*(y-z).
\end{equation}
\vl{This fits the master viewpoint after introducing an auxiliary observation-space variable. Its row-action realization uses an additional correction step, so we include it in the modeling framework while keeping the main convergence theorem focused on Algorithm~\ref{alg:hard}.}

\paragraph{Simplex feasibility}
The simplex can be enforced in two ways. In a primal Euclidean realization, one projects onto $\Delta_n$. In an entropic realization, the mirror map already takes values in $\Delta_n$. The model is the same, but the local geometry and the update are different.

\subsection{From a unified formulation to two algorithmic realizations}
The soft and hard models differ at the target level. Under the usual convexity and qualification assumptions, a minimizer of the soft data-fitting problem satisfies
\begin{equation}
\label{eq:soft-optimality}
    0\in A^\top\partial_1D(Ax^\star,y)+\partial r(x^\star)+N_C(x^\star).
\end{equation}
There is no reason for $Ax^\star=y$. A sampled row therefore gives a stochastic direction for the objective.

In the selector problem \eqref{eq:selector}, every row equation is a hard constraint. A row update can be derived from a Bregman projection or a coordinate update on the Fenchel dual. The selector geometry determines which feasible point is recovered.

This leads to two branches.
\begin{itemize}
    \item \textbf{Branch I: soft constrained data fitting.} Use a primal row oracle and a Bregman-proximal step. This branch allows a nonzero optimal residual.
    \item \textbf{Branch II: hard equality selectors.} Update a dual variable row by row and recover the primal iterate through $\nabla f^*$. This branch uses exact feasibility in its analysis.
\end{itemize}

\vl{The same formulation also shows when additional ingredients are needed for an efficient row-action method.}
\begin{enumerate}
    \item \vl{Local row or block access requires a corresponding decomposition of $D$; otherwise the oracle may still involve the full matrix.}
    \item \vl{The Bregman-proximal subproblem should be inexpensive enough to preserve the advantage of a local row update.}
    \item \vl{For underdetermined nonnegative selectors, the local result in Section~\ref{subsec:local-geometries} gives a computable contraction once the iterates enter the stated neighborhood. A global computable constant may require additional regularity.}
\end{enumerate}

\subsection{Branch I: primal Bregman-proximal row action}
Let $h$ be differentiable and strongly convex. Given a row oracle $s_k$, a stepsize $\eta_k>0$, and $R=r+\ind_C$, define
\begin{equation}
\label{eq:soft-update}
    x_{k+1}
    =
    \argmin_{x\in\RR^n}
    \left\{
        \eta_k\ip{s_k}{x}
        +D_h(x,x_k)
        +\eta_k R(x)
    \right\}.
\end{equation}
The row oracle comes from the data term. The Bregman distance controls the geometry. The function $R$ enforces the constraint and the regularization.

\begin{algorithm}[H]
\caption{Primal Bregman-proximal row-action method}
\label{alg:soft}
\begin{algorithmic}[1]
\Require $x_0\in\dom R$, stepsizes $\{\eta_k\}$, row or block sampling rule
\For{$k=0,1,\ldots$}
    \State Sample a row or block index $i_k$
    \State Build a local oracle $s_k=s_{i_k}(x_k)$ from the data term
    \State Compute $x_{k+1}$ from \eqref{eq:soft-update}
\EndFor
\end{algorithmic}
\end{algorithm}

For $h(x)=\frac12\norm{x}_2^2$ and $r=0$, the update reduces to \textcolor{black}{\mbox{$x_{k+1}=P_C(x_k-\eta_ks_k)$}}.
For least squares, the normalized row oracle is
\begin{equation}
\label{eq:normalized-row-oracle}
    g_i(x)=\frac{a_i^\top x-y_i}{\norm{a_i}_2^2}a_i.
\end{equation}
With entropy geometry on the simplex, \eqref{eq:soft-update} becomes a multiplicative update followed by normalization. This normalization is the closed-form Bregman-proximal map on the simplex, not an additional Euclidean projection.

The generic update also covers blocks. The local oracle can be obtained from $A_{I,:}$ and $y_I$, and the Bregman-proximal step remains in the variable space. The theory below is written for rows to keep the notation light.

\subsection{Branch II: dual Bregman--Kaczmarz selectors}
Consider \eqref{eq:selector}, where $f$ is proper, closed, and $\mu_f$-strongly convex. Its Fenchel dual can be written as
\[
    \max_{v\in\RR^m} y^\top v-f^*(A^\top v).
\]
Set $u=A^\top v$. The primal variable is recovered through \textcolor{black}{\mbox{$x=\nabla f^*(u)$}}.
A relaxed row step is
\begin{equation}
\label{eq:selector-update}
    u_{k+1}=u_k-\eta_k g_{i_k}(x_k),
    \qquad
    x_{k+1}=\nabla f^*(u_{k+1}),
\end{equation}
where $g_i$ is defined by \eqref{eq:normalized-row-oracle}, with $y$ replaced by $b$ when the system is written as $Ax=b$.

\begin{algorithm}[H]
\caption{Relaxed randomized Bregman--Kaczmarz method}
\label{alg:hard}
\begin{algorithmic}[1]
\Require $u_0\in\range(A^\top)$, $x_0=\nabla f^*(u_0)$, stepsizes $\{\eta_k\}$
\For{$k=0,1,\ldots$}
    \State Sample $i_k$ with probability $p_i=\norm{a_i}_2^2/\norm{A}_F^2$
    \State $g_{i_k}(x_k)\gets (a_{i_k}^\top x_k-b_{i_k})a_{i_k}/\norm{a_{i_k}}_2^2$
    \State $u_{k+1}\gets u_k-\eta_kg_{i_k}(x_k)$
    \State $x_{k+1}\gets\nabla f^*(u_{k+1})$
\EndFor
\end{algorithmic}
\end{algorithm}

The terminal map determines the structure:
\begin{align*}
    f(x)=\tfrac12\norm{x}_2^2
    &\quad\Longrightarrow\quad \nabla f^*(u)=u,\\
    f(x)=\lambda\norm{x}_1+\tfrac12\norm{x}_2^2
    &\quad\Longrightarrow\quad \nabla f^*(u)=S_\lambda(u),\\
    f(x)=\ind_{\RR_+^n}(x)+\tfrac12\norm{x}_2^2
    &\quad\Longrightarrow\quad \nabla f^*(u)=\relu(u),\\
    f(x)=\sum_jx_j\log x_j+\ind_{\Delta_n}(x)
    &\quad\Longrightarrow\quad \nabla f^*(u)=\softmax(u).
\end{align*}
Thus the identity, soft-thresholding, ReLU, and softmax maps appear in the same dual row-action scheme.

An exact Bregman projection onto one row hyperplane can also be obtained by a one-dimensional line search. The relaxed update \eqref{eq:selector-update} avoids this line search. In Euclidean geometry, it gives the exact row projection for $\eta_k=1$ and the usual relaxed Kaczmarz step otherwise. It remains simple for general mirror maps.

\section{Convergence analysis}
\label{sec:theory}
The analysis follows the two algorithmic branches. We begin with the common inequality behind the primal Bregman-proximal step. We then consider the soft branch at three levels.
\begin{itemize}
    \item Theorem~\ref{thm:det-soft} gives a deterministic $O(1/K)$ reference result under relative smoothness.
    \item Theorem~\ref{thm:stoch-soft} treats an unbiased stochastic row oracle and gives the usual deterministic-plus-variance bound.
    \item Theorems~\ref{thm:ls-global} and~\ref{thm:ls-linear} use the exact structure of the least-squares row oracle. They cover, respectively, noisy constrained fitting and the interpolating regime.
\end{itemize}
We then turn to the hard equality branch. Theorem~\ref{thm:selector-contraction} converts a residual-to-Bregman error bound into linear convergence. Theorem~\ref{thm:global-eb-sufficient} recalls global sufficient conditions from the extended Bregman--Kaczmarz literature. Finally, Propositions~\ref{prop:local-nonnegative} and~\ref{prop:local-simplex} give local constants for the ReLU and softmax geometries.

\paragraph{Recovery of existing results}
The analysis below also provides several consistency checks with the literature. In Euclidean geometry, the soft branch becomes a projected randomized row method. If the constraint is removed and the linear system is consistent, the update and its rate reduce to the classical randomized Kaczmarz method. In the hard branch, a quadratic potential gives the same Euclidean method, while an elastic-net potential gives randomized sparse Kaczmarz. We make these reductions explicit after the corresponding theorems; see also~\cite{StrohmerVershynin2009,WuXiang2020PRK,schopfer2019linear}.

\subsection{A common Bregman-proximal inequality}
Assume throughout this subsection and the next that $h$ is differentiable and $\mu_h$-strongly convex with respect to a norm $\norm{\cdot}$. Let $R$ be proper, closed, and convex, and assume that the subproblem \eqref{eq:soft-update} has a solution.

\begin{proposition}[One-step Bregman-proximal inequalities]
\label{prop:one-step}
For every $x\in\dom R$, the update \eqref{eq:soft-update} satisfies
\begin{align}
\label{eq:exact-one-step}
    &\eta_k\ip{s_k}{x_{k+1}-x}
    +\eta_k\bigl(R(x_{k+1})-R(x)\bigr)\\
    &\qquad\le
    D_h(x,x_k)-D_h(x,x_{k+1})-D_h(x_{k+1},x_k).\notag
\end{align}
It also satisfies
\begin{align}
\label{eq:relaxed-one-step}
    &\eta_k\ip{s_k}{x_k-x}
    +\eta_k\bigl(R(x_{k+1})-R(x)\bigr)\\
    &\qquad\le
    D_h(x,x_k)-D_h(x,x_{k+1})
    +\frac{\eta_k^2}{2\mu_h}\norm{s_k}_*^2.\notag
\end{align}
\end{proposition}

\begin{proof}
The optimality condition of \eqref{eq:soft-update} gives a vector $q_{k+1}\in\partial R(x_{k+1})$ such that
\[
    \eta_ks_k+\nabla h(x_{k+1})-\nabla h(x_k)+\eta_kq_{k+1}=0.
\]
The subgradient inequality gives \textcolor{black}{$R(x_{k+1})-R(x)\le\ip{q_{k+1}}{x_{k+1}-x}$}.
Substitution and the three-point identity \eqref{eq:three-point} yield \eqref{eq:exact-one-step}.

To obtain \eqref{eq:relaxed-one-step}, write \textcolor{black}{$\ip{s_k}{x_k-x}=\ip{s_k}{x_{k+1}-x}+\ip{s_k}{x_k-x_{k+1}}$}.
Strong convexity and Young's inequality give
\[
    \eta_k\ip{s_k}{x_k-x_{k+1}}-D_h(x_{k+1},x_k)
    \le\frac{\eta_k^2}{2\mu_h}\norm{s_k}_*^2.
\]
Combining this bound with \eqref{eq:exact-one-step} proves the second inequality.
\end{proof}

This proposition is the geometric backbone of Branch I. It does not yet use the origin of $s_k$. The next results add objective structure.

\subsection{General soft data fitting}
\label{subsec:general-soft}
Let
\begin{equation}
\label{eq:general-soft-problem}
    \min_x F(x):=\Phi(x)+R(x),
\end{equation}
where $\Phi$ is convex and differentiable. We use relative smoothness because it matches the Bregman geometry \cite{LuFreundNesterov2018}.

\begin{assumption}[Relative smoothness]
\label{ass:relative-smoothness}
There exists $L_\Phi>0$ such that, for all $x,z$ in the domain,
\begin{equation}
\label{eq:relative-smoothness}
    \Phi(z)\le\Phi(x)+\ip{\nabla\Phi(x)}{z-x}+L_\Phi D_h(z,x).
\end{equation}
\end{assumption}
In Euclidean geometry, this is the usual Lipschitz-gradient inequality.

The deterministic result is a reference statement for the full-gradient version of \eqref{eq:soft-update}. It is included because it shows the role of the geometry without stochastic noise. The randomized row-action result follows next.

\begin{theorem}[Deterministic soft branch]
\label{thm:det-soft}
Assume \ref{ass:relative-smoothness} and set $s_k=\nabla\Phi(x_k)$. If $0<\eta_k\le1/L_\Phi$, then, for every $x\in\dom R$,
\begin{equation}
\label{eq:det-soft-sum}
    \sum_{k=0}^{K-1}\eta_k\bigl(F(x_{k+1})-F(x)\bigr)
    \le D_h(x,x_0).
\end{equation}
Let \textcolor{black}{$S_K=\sum_{k=0}^{K-1}\eta_k$ and $\bar x_K=\tfrac{1}{S_K}\sum_{k=0}^{K-1}\eta_kx_{k+1}$}.
Then
\begin{equation}
\label{eq:det-soft-rate}
    F(\bar x_K)-F(x)\le\frac{D_h(x,x_0)}{S_K}.
\end{equation}
In particular, a constant stepsize gives an $O(1/K)$ bound.
\end{theorem}

The proof is given in Appendix~\ref{app:general-soft-proofs}. It combines \eqref{eq:exact-one-step} with relative smoothness and then telescopes. 
\vl{This theorem is the deterministic reference for the stochastic row-action analysis and shows explicitly how the Bregman geometry enters the bound.}

For a row-action method, the local oracle must be scaled so that it is unbiased for the full gradient. We write
\begin{equation}
\label{eq:stoch-oracle-general}
    s_k=\nabla\Phi(x_k)+\delta_k,
    \qquad
    \EE[\delta_k\mid\cF_k]=0.
\end{equation}

\begin{assumption}[Bounded conditional variance]
\label{ass:bounded-variance}
There exists $\sigma\ge0$ such that
\[
    \EE[\norm{\delta_k}_*^2\mid\cF_k]\le\sigma^2
    \qquad\text{a.s. for every }k.
\]
\end{assumption}

\begin{theorem}[Stochastic soft branch]
\label{thm:stoch-soft}
Assume \ref{ass:relative-smoothness} and \ref{ass:bounded-variance}. Let the oracle satisfy \eqref{eq:stoch-oracle-general}, and assume \textcolor{black}{$0<\eta_k\le\bar\eta<\tfrac{1}{L_\Phi}$}.
Then, for every $x\in\dom R$,
\begin{align}
\label{eq:stoch-soft-sum}
    \sum_{k=0}^{K-1}\eta_k\EE[F(x_{k+1})-F(x)]
    \le
    D_h(x,x_0)
    +\frac{\sigma^2}{2\mu_h(1-\bar\eta L_\Phi)}
    \sum_{k=0}^{K-1}\eta_k^2.
\end{align}
For the weighted average \textcolor{black}{$\bar x_K=\tfrac{1}{S_K}\sum_{k=0}^{K-1}\eta_kx_{k+1}$}, we have
\begin{equation}
\label{eq:stoch-soft-rate}
    \EE[F(\bar x_K)-F(x)]
    \le
    \frac{D_h(x,x_0)}{S_K}
    +\frac{\sigma^2}{2\mu_h(1-\bar\eta L_\Phi)}
    \frac{\sum_{k=0}^{K-1}\eta_k^2}{S_K}.
\end{equation}
Thus a horizon-tuned rule $\eta_k=\eta_0/\sqrt K$ gives an $O(K^{-1/2})$ bound. If $\sigma=0$, the deterministic $O(1/S_K)$ bound is recovered.
\end{theorem}

The proof is given in Appendix~\ref{app:general-soft-proofs}. The ratio \textcolor{black}{$\tfrac{\sum_k\eta_k^2}{\sum_k\eta_k}$} is the usual stochastic term. The constrained least-squares oracle has a more precise second-moment identity. We use it next instead of a uniform variance bound.

\begin{remark}[Euclidean and mirror special cases]
If
\[
    h(x)=\frac12\norm{x}_2^2
    \qquad\text{and}\qquad
    R=\ind_C,
\]
the update \eqref{eq:soft-update} becomes a projected stochastic-gradient step, and Theorem~\ref{thm:stoch-soft} gives its standard ergodic convergence bound. If $R=0$ and $h$ is a general strongly convex reference function, the same result becomes the corresponding stochastic mirror-descent estimate. In the present framework, the stochastic oracle is generated from one row or one block of the linear observation model.
\end{remark}

\subsection{Noisy constrained least squares}
\label{subsec:noisy-cls-theory}
Consider
\begin{equation}
\label{eq:cls}
    \min_{x\in C}\Phi(x),
    \qquad
    \Phi(x)=\frac12\norm{Ax-y}_2^2,
\end{equation}
where $C$ is nonempty, closed, and convex. Sample row $i$ with probability
\begin{equation}
\label{eq:row-prob}
    p_i=\frac{\norm{a_i}_2^2}{\norm{A}_F^2},
\end{equation}
and use the normalized row oracle \eqref{eq:normalized-row-oracle}. Branch I becomes
\begin{equation}
\label{eq:projected-cls}
    x_{k+1}=P_C\bigl(x_k-\eta_kg_{i_k}(x_k)\bigr).
\end{equation}

The following lemma is the key difference with the generic stochastic theorem. The conditional second moment is not bounded by an arbitrary constant. It is exactly proportional to the current objective.

\begin{lemma}[Exact oracle identities]
\label{lem:ls-identities}
For every $x\in\RR^n$,
\begin{equation}
\label{eq:unbiased}
    \EE_i[g_i(x)]=\frac{1}{\norm{A}_F^2}\nabla\Phi(x),
\end{equation}
and
\begin{equation}
\label{eq:second-moment}
    \EE_i\bigl[\norm{g_i(x)}_2^2\bigr]
    =\frac{1}{\norm{A}_F^2}\norm{Ax-y}_2^2
    =\frac{2}{\norm{A}_F^2}\Phi(x).
\end{equation}
\end{lemma}

\begin{proof}
Using \eqref{eq:row-prob},
\[
    \EE_i[g_i(x)]
    =\sum_i\frac{\norm{a_i}_2^2}{\norm{A}_F^2}
    \frac{a_i^\top x-y_i}{\norm{a_i}_2^2}a_i
    =\frac{A^\top(Ax-y)}{\norm{A}_F^2}.
\]
Moreover, \textcolor{black}{$\norm{g_i(x)}_2^2=\tfrac{(a_i^\top x-y_i)^2}{\norm{a_i}_2^2}$}.
Taking the expectation with the same probabilities proves \eqref{eq:second-moment}.
\end{proof}

\begin{proposition}[Conditional recursion]
\label{prop:cls-recursion}
Let $\beta=\norm{A}_F^{-2}$. For every $x\in C$,
\begin{align}
\label{eq:cls-recursion}
    \eta_k\beta\bigl(\Phi(x_k)-\Phi(x)\bigr)
    &\le
    \frac12\norm{x-x_k}_2^2
    -\EE\left[\frac12\norm{x-x_{k+1}}_2^2\mid\cF_k\right]
    +\eta_k^2\beta\Phi(x_k).
\end{align}
Equivalently,
\begin{align}
\label{eq:cls-recursion-second}
    \eta_k\beta(1-\eta_k)\Phi(x_k)
    &\le
    \frac12\norm{x-x_k}_2^2
    -\EE\left[\frac12\norm{x-x_{k+1}}_2^2\mid\cF_k\right]
    +\eta_k\beta\Phi(x).
\end{align}
\end{proposition}

\begin{proof}
Apply \eqref{eq:relaxed-one-step} with $h(x)=\frac12\norm{x}_2^2$, $R=\ind_C$, and $s_k=g_{i_k}(x_k)$. Since $x,x_{k+1}\in C$, the terms involving $R$ vanish. Taking conditional expectation and using \eqref{eq:unbiased} gives
\begin{align*}
    \eta_k\beta\ip{\nabla\Phi(x_k)}{x_k-x}
    &\le
    \frac12\norm{x-x_k}_2^2
    -\EE\left[\frac12\norm{x-x_{k+1}}_2^2\mid\cF_k\right]\\
    &\quad +\frac{\eta_k^2}{2}\EE[\norm{g_{i_k}(x_k)}_2^2\mid\cF_k].
\end{align*}
Convexity of $\Phi$ and \eqref{eq:second-moment} give \eqref{eq:cls-recursion}. Rearranging the last term gives \eqref{eq:cls-recursion-second}.
\end{proof}

The next theorem treats the general inconsistent case. The weighted average is essential in the proof because the row oracle does not vanish at a noninterpolating minimizer.

\begin{theorem}[Expected ergodic bound for constrained least squares]
\label{thm:ls-global}
Let $x^\star$ be a minimizer of \eqref{eq:cls}. Assume \textcolor{black}{$0<\eta_k\le\bar\eta<1$}.
Define \textcolor{black}{$S_K=\sum_{k=0}^{K-1}\eta_k$ and $\bar x_K=\tfrac{1}{S_K}\sum_{k=0}^{K-1}\eta_kx_k$}.
Then
\begin{align}
\label{eq:global-bound}
    \EE[\Phi(\bar x_K)-\Phi(x^\star)]
    &\le
    \frac{\norm{A}_F^2\norm{x^\star-x_0}_2^2}{2(1-\bar\eta)S_K}
    +\frac{\Phi(x^\star)}{1-\bar\eta}
    \frac{\sum_{k=0}^{K-1}\eta_k^2}{S_K}.
\end{align}
Consequently:
\begin{enumerate}
    \item If $\Phi(x^\star)=0$ and $\eta_k\equiv\eta\in(0,1)$, then
    \[
        \EE[\Phi(\bar x_K)]
        \le
        \frac{\norm{A}_F^2\norm{x^\star-x_0}_2^2}{2(1-\eta)\eta K}.
    \]
    \item In the general case, the horizon rule $\eta_k\equiv\eta_0/\sqrt K$ gives an $O(K^{-1/2})$ bound.
    \item A constant stepsize $\eta\in(0,1)$ gives the neighborhood term \textcolor{black}{$\tfrac{\eta}{1-\eta}\Phi(x^\star)$}.
\end{enumerate}
\end{theorem}

\begin{proof}
Apply \eqref{eq:cls-recursion-second} with $x=x^\star$ and subtract $\eta_k\beta(1-\eta_k)\Phi(x^\star)$ from both sides. This gives
\begin{align*}
    \eta_k\beta(1-\eta_k)\bigl(\Phi(x_k)-\Phi(x^\star)\bigr)
    &\le
    \frac12\norm{x^\star-x_k}_2^2
    -\EE\left[\frac12\norm{x^\star-x_{k+1}}_2^2\mid\cF_k\right]
    +\eta_k^2\beta\Phi(x^\star).
\end{align*}
Take total expectation and sum. The distance terms telescope. Since $1-\eta_k\ge1-\bar\eta$,
\[
    \beta(1-\bar\eta)\sum_{k=0}^{K-1}\eta_k
    \EE[\Phi(x_k)-\Phi(x^\star)]
    \le
    \frac12\norm{x^\star-x_0}_2^2
    +\beta\Phi(x^\star)\sum_{k=0}^{K-1}\eta_k^2.
\]
Divide by $\beta(1-\bar\eta)S_K$ and use convexity of $\Phi$ at $\bar x_K$.
\end{proof}

The bound separates two terms. The first is the deterministic optimization error. The second is caused by the nonzero residual. A fixed stepsize does not remove this second term. This is why the numerical test in Section~\ref{subsec:noisy-nnls} compares fixed and decreasing stepsizes.

\subsection{Linear convergence in the interpolating case}
\label{subsec:interpolating}
We now assume that
\[
    \cX^\star=\{x\in C:Ax=y\}
\]
is nonempty. In this regime, the normalized row oracle satisfies an exact identity, so the actual iterates can be studied directly.

\begin{assumption}[Residual error bound on $C$]
\label{ass:residual-eb}
There exists $\sigma_C\in(0,\norm{A}_F]$ such that
\begin{equation}
\label{eq:residual-eb}
    \norm{Ax-y}_2^2\ge\sigma_C^2\dist(x,\cX^\star)^2
    \qquad\text{for every }x\in C.
\end{equation}
\end{assumption}
The upper restriction causes no loss. If the bound holds for a larger constant, it also holds after decreasing it.

\begin{lemma}[Row-wise exactness]
\label{lem:row-exactness}
For every $x^\star\in\cX^\star$ and every row $i$,
\begin{equation}
\label{eq:row-exactness}
    \ip{g_i(x)}{x-x^\star}=\norm{g_i(x)}_2^2.
\end{equation}
\end{lemma}

\begin{proof}
Since $a_i^\top x^\star=y_i$,
\[
    \ip{g_i(x)}{x-x^\star}
    =\frac{(a_i^\top x-y_i)a_i^\top(x-x^\star)}{\norm{a_i}_2^2}
    =\frac{(a_i^\top x-y_i)^2}{\norm{a_i}_2^2}.
\]
\end{proof}

\begin{theorem}[Linear convergence of the projected iterates]
\label{thm:ls-linear}
Assume that $\cX^\star\neq\varnothing$ and Assumption~\ref{ass:residual-eb} holds. Let
\[
    x_{k+1}=P_C\bigl(x_k-\eta g_{i_k}(x_k)\bigr),
    \qquad 0<\eta<2,
\]
with probabilities \eqref{eq:row-prob}. Then
\begin{equation}
\label{eq:linear-rate}
    \EE[\dist(x_{k+1},\cX^\star)^2]
    \le
    \left(1-\frac{\eta(2-\eta)\sigma_C^2}{\norm{A}_F^2}\right)
    \EE[\dist(x_k,\cX^\star)^2].
\end{equation}
The factor is minimized at $\eta=1$.
\end{theorem}

\begin{proof}
Let $x_k^\star$ be a Euclidean projection of $x_k$ onto $\cX^\star$. Since $x_k^\star\in C$, nonexpansiveness of $P_C$ gives
\begin{align*}
    \dist(x_{k+1},\cX^\star)^2
    &\le \norm{x_k-\eta g_{i_k}(x_k)-x_k^\star}_2^2\\
    &=\dist(x_k,\cX^\star)^2
    -2\eta\ip{g_{i_k}(x_k)}{x_k-x_k^\star}
    +\eta^2\norm{g_{i_k}(x_k)}_2^2.
\end{align*}
Lemma~\ref{lem:row-exactness} gives
\[
    \dist(x_{k+1},\cX^\star)^2
    \le
    \dist(x_k,\cX^\star)^2
    -\eta(2-\eta)\norm{g_{i_k}(x_k)}_2^2.
\]
Take conditional expectation, use \eqref{eq:second-moment}, and apply \eqref{eq:residual-eb}. Taking total expectation proves \eqref{eq:linear-rate}.
\end{proof}

\begin{corollary}[Classical randomized Kaczmarz rate]
\label{cor:classical-rate}
Let $C=\RR^n$, assume $Ax=y$ is consistent, and take $x_0\in\range(A^\top)$. Let $x^\dagger$ be the minimum-norm solution. Then
\[
    \EE[\norm{x_{k+1}-x^\dagger}_2^2]
    \le
    \left(1-\frac{\eta(2-\eta)(\sigma_{\min}^+(A))^2}{\norm{A}_F^2}\right)
    \EE[\norm{x_k-x^\dagger}_2^2].
\]
For $\eta=1$, this is the standard randomized Kaczmarz factor \cite{StrohmerVershynin2009}.
\end{corollary}

The role of this corollary is an exact consistency check. When both the constraint and the nonzero optimal residual disappear, the framework recovers the classical randomized Kaczmarz update and its standard convergence factor. More generally, for $\eta=1$ and a closed convex set $C$, the update \eqref{eq:projected-cls} is the projected randomized Kaczmarz method considered, for example, in~\cite{Popa2008CKE,WuXiang2020PRK}. Theorem~\ref{thm:ls-linear} shows that the same row-wise identity gives linear convergence whenever the residual error bound \eqref{eq:residual-eb} holds on $C$.

\subsection{Equality-constrained Bregman selectors}
\label{subsec:selector-theory}
Consider
\begin{equation}
\label{eq:hard-problem}
    \min_{x\in\RR^n} f(x)
    \qquad\text{subject to}\qquad
    Ax=b,
\end{equation}
where $f$ is proper, closed, and $\mu_f$-strongly convex. We assume that the problem is feasible. The constrained minimizer $\hat x$ is then unique. For $u\in\partial f(x)$, the generalized Bregman distance is
\begin{equation}
\label{eq:generalized-bregman}
    D_f^u(x,\hat x)
    =f(\hat x)-f(x)-\ip{u}{\hat x-x}
    =f^*(u)-\ip{u}{\hat x}+f(\hat x).
\end{equation}

The following assumption converts the feasibility residual into distance in the selector geometry.

\begin{assumption}[Residual-to-Bregman error bound]
\label{ass:bregman-eb}
There exists $\theta(\hat x)>0$ such that
\begin{equation}
\label{eq:bregman-eb}
    D_f^u(x,\hat x)
    \le\frac{1}{\theta(\hat x)}\norm{Ax-b}_2^2
\end{equation}
for every pair $x,u$ with $u\in\partial f(x)\cap\range(A^\top)$ considered by the algorithm.
\end{assumption}

The condition is not automatic. It is the selector analogue of the residual error bound used in Theorem~\ref{thm:ls-linear}. The main difference is that the distance depends on the geometry and, for nonsmooth $f$, on the selected subgradient $u$.

\begin{theorem}[Expected contraction of relaxed Bregman--Kaczmarz]
\label{thm:selector-contraction}
Assume that $f$ is $\mu_f$-strongly convex and that Assumption~\ref{ass:bregman-eb} holds. Let $u_0\in\range(A^\top)$, set $x_k=\nabla f^*(u_k)$, and apply Algorithm~\ref{alg:hard}. Assume \textcolor{black}{$0<\eta_k<2\mu_f$}.
Let $\bar\theta\in(0,\theta(\hat x)]$ be chosen so that \textcolor{black}{$\rho_k:=1-\tfrac{\bar\theta}{\norm{A}_F^2}\left(\eta_k-\tfrac{\eta_k^2}{2\mu_f}\right)\in[0,1)$}.
Then
\begin{equation}
\label{eq:selector-contraction}
    \EE[D_f^{u_k}(x_k,\hat x)]
    \le
    \left(\prod_{j=0}^{k-1}\rho_j\right)
    D_f^{u_0}(x_0,\hat x).
\end{equation}
If the stepsizes stay in a compact subinterval of $(0,2\mu_f)$, the convergence is linear in expectation with a uniform factor.
\end{theorem}

\begin{proof}
Since $f$ is $\mu_f$-strongly convex, $f^*$ has a $1/\mu_f$-Lipschitz gradient. Hence
\[
    f^*(u_{k+1})
    \le f^*(u_k)+\ip{x_k}{u_{k+1}-u_k}
    +\frac{1}{2\mu_f}\norm{u_{k+1}-u_k}_2^2.
\]
Using \eqref{eq:generalized-bregman} and $u_{k+1}-u_k=-\eta_kg_{i_k}(x_k)$ gives
\begin{align*}
    D_f^{u_{k+1}}(x_{k+1},\hat x)
    &\le D_f^{u_k}(x_k,\hat x)
    -\eta_k\ip{g_{i_k}(x_k)}{x_k-\hat x}
    +\frac{\eta_k^2}{2\mu_f}\norm{g_{i_k}(x_k)}_2^2.
\end{align*}
Because $A\hat x=b$, Lemma~\ref{lem:row-exactness} gives
\[
    D_f^{u_{k+1}}(x_{k+1},\hat x)
    \le D_f^{u_k}(x_k,\hat x)
    -\left(\eta_k-\frac{\eta_k^2}{2\mu_f}\right)
    \norm{g_{i_k}(x_k)}_2^2.
\]
Taking conditional expectation and using \eqref{eq:second-moment}, with $y=b$, yields
\[
    \EE[D_f^{u_{k+1}}(x_{k+1},\hat x)\mid\cF_k]
    \le D_f^{u_k}(x_k,\hat x)
    -\left(\eta_k-\frac{\eta_k^2}{2\mu_f}\right)
    \frac{\norm{Ax_k-b}_2^2}{\norm{A}_F^2}.
\]
Assumption~\ref{ass:bregman-eb}, weakened from $\theta$ to $\bar\theta$, gives
\[
    \norm{Ax_k-b}_2^2\ge\bar\theta D_f^{u_k}(x_k,\hat x).
\]
The conditional contraction follows. Taking total expectation and iterating proves \eqref{eq:selector-contraction}.
\end{proof}

For a constant stepsize, the scalar quantity \textcolor{black}{$\eta-\tfrac{\eta^2}{2\mu_f}$} is maximized at $\eta=\mu_f$. This is the best formal choice in the bound. The practical choice can still depend on the tightness of the error-bound constant.

\begin{remark}[Euclidean selector and classical randomized Kaczmarz]
Consider \textcolor{black}{$f(x)=\tfrac12\norm{x}_2^2$}.
Then $\mu_f=1$, $\nabla f^*(u)=u$, and \textcolor{black}{$D_f^u(x,\hat x)=\tfrac12\norm{x-\hat x}_2^2$}.
Algorithm~\ref{alg:hard} therefore becomes the relaxed randomized Kaczmarz method, and for $\eta_k=1$ it is the classical randomized Kaczmarz method. Moreover, on $\range(A^\top)$,
\[
    \norm{Ax-b}_2^2
    \ge
    2\bigl(\sigma_{\min}^+(A)\bigr)^2
    D_f^u(x,\hat x).
\]
Hence Assumption~\ref{ass:bregman-eb} holds with \textcolor{black}{$\theta(\hat x)=2\bigl(\sigma_{\min}^+(A)\bigr)^2$}.
Taking $\eta_k=1$ in Theorem~\ref{thm:selector-contraction} gives
\[
    \EE\!\left[\norm{x_k-\hat x}_2^2\right]
    \le
    \left(
        1-\frac{\bigl(\sigma_{\min}^+(A)\bigr)^2}
        {\norm{A}_F^2}
    \right)^k
    \norm{x_0-\hat x}_2^2,
\]
which is the classical randomized Kaczmarz rate~\cite{StrohmerVershynin2009}. Thus the two algorithmic branches coincide in this Euclidean consistent case.
\end{remark}

\subsection{Global sufficient conditions for the selector error bound}
\label{subsec:global-selector-eb}
We now recall conditions that make Assumption~\ref{ass:bregman-eb} global. They are useful because they show that the contraction theorem is not only a local statement for smooth geometries.

\begin{definition}[Calm subdifferential]
The subdifferential mapping $\partial f$ is calm at $\hat x$ if there exist $c,L>0$ such that
\[
    \partial f(x)\subseteq\partial f(\hat x)+L\norm{x-\hat x}_2B_2
\]
whenever $\norm{x-\hat x}_2\le c$.
\end{definition}

\begin{definition}[Linear regularity]
Assume $\partial f(\hat x)\cap\range(A^\top)\neq\varnothing$. The pair $\{\partial f(\hat x),\range(A^\top)\}$ is linearly regular if there exists $\zeta>0$ such that, for every $u\in\RR^n$,
\[
    \dist\bigl(u,\partial f(\hat x)\cap\range(A^\top)\bigr)
    \le\zeta\Bigl(
    \dist(u,\partial f(\hat x))
    +\dist(u,\range(A^\top))
    \Bigr).
\]
\end{definition}

\begin{definition}[At-most-linear growth]
The mapping $\partial f$ grows at most linearly if there exist $\rho_1,\rho_2\ge0$ such that
\[
    \norm{u}_2\le\rho_1\norm{x}_2+\rho_2
    \qquad\text{for every }x\text{ and }u\in\partial f(x).
\]
\end{definition}

\begin{theorem}[Global residual-to-Bregman bound \cite{Schpfer2022ExtendedRK}]
\label{thm:global-eb-sufficient}
Let $f:\RR^n\to\RR$ be strongly convex. Assume that $\partial f$ grows at most linearly, is calm at the unique solution $\hat x$ of \eqref{eq:hard-problem}, and that $\{\partial f(\hat x),\range(A^\top)\}$ is linearly regular. Then there exists $\theta(\hat x)>0$ such that
\[
    D_f^u(x,\hat x)
    \le\frac{1}{\theta(\hat x)}\norm{Ax-b}_2^2
\]
for every $x\in\RR^n$ and every $u\in\partial f(x)\cap\range(A^\top)$.
\end{theorem}

The invariant required in this theorem is preserved by Algorithm~\ref{alg:hard}. Indeed, $x_k=\nabla f^*(u_k)$ implies $u_k\in\partial f(x_k)$, and every dual update adds a multiple of a row of $A$. Hence $u_k\in\range(A^\top)$ for all $k$ if it holds at initialization.

For the sparse selector
\[
    f(x)=\lambda\norm{x}_1+\frac12\norm{x}_2^2,
\]
the mirror map is soft-thresholding. The regularity conditions hold, and an explicit form of the constant is available \cite{schopfer2019linear}. If $|\hat x|_{\min}$ denotes the smallest nonzero magnitude of $\hat x$, then
\begin{equation}
\label{eq:sparse-theta}
    \frac{1}{\theta(\hat x)}
    =\frac{1}{\widetilde\sigma_{\min}(A)^2}
    \frac{|\hat x|_{\min}+2\lambda}{|\hat x|_{\min}},
\end{equation}
where
\[
    \widetilde{\sigma}_{\min}(A)
    :=
    \min\left\{
        \sigma_{\min}^{+}(A_{:,J})
        \;\middle|\;
        J\subseteq\{1,\ldots,n\},
        \ A_{:,J}\neq 0
    \right\},
\]
and $\sigma_{\min}^{+}$ denotes the smallest positive singular value, appearing in the sparse Kaczmarz analysis.

\begin{corollary}[Randomized sparse Kaczmarz rate]
\label{cor:sparse-kaczmarz-rate}
Let
\[
    f(x)=\lambda\norm{x}_1+\frac12\norm{x}_2^2,
\]
and assume that the conditions of Theorem~\ref{thm:global-eb-sufficient} hold. Let $\hat x\neq0$ be the solution of \eqref{eq:hard-problem}, and apply Algorithm~\ref{alg:hard} with $\eta_k=1$. Then
\[
    \EE\!\left[D_f^{u_k}(x_k,\hat x)\right]
    \le
    \left(
        1-
        \frac{\widetilde\sigma_{\min}(A)^2}
        {2\norm{A}_F^2}
        \frac{|\hat x|_{\min}}
        {|\hat x|_{\min}+2\lambda}
    \right)^k
    D_f^{u_0}(x_0,\hat x).
\]
This is the linear convergence factor obtained for randomized sparse Kaczmarz in~\cite{schopfer2019linear}.
\end{corollary}

Theorem~\ref{thm:global-eb-sufficient} is stated for real-valued $f$. The nonnegative quadratic and entropy potentials used below are extended-valued because they contain an indicator. We therefore derive local geometry-specific constants directly.

\subsection{Local constants for nonnegative and simplex geometries}
\label{subsec:local-geometries}
\vl{We now give explicit local error-bound constants for the nonnegative and simplex geometries. They yield selector contractions once the iterates are in the stated neighborhood.}

\begin{proposition}[Local nonnegative quadratic error bound]
\label{prop:local-nonnegative}
Let
\[
    f(x)=\ind_{\RR_+^n}(x)+\frac12\norm{x}_2^2,
    \qquad
    \nabla f^*(u)=\relu(u).
\]
Let $\hat x$ solve \eqref{eq:hard-problem}, and set \textcolor{black}{$S=\supp(\hat x)$ and $\cT_S=\{d\in\RR^n:d_{S^c}\ge0\}$}.
Assume \textcolor{black}{$\ker(A)\cap\cT_S=\{0\}$}, and define \textcolor{black}{$\sigma_S^+(A)=\inf_{\substack{d\in\cT_S\\\norm{d}_2=1}}\norm{Ad}_2>0$}.
If $x\ge0$, $u\in\partial f(x)$, and $x_S>0$, then
\begin{equation}
\label{eq:local-nonnegative}
    D_f^u(x,\hat x)
    \le
    \frac{1}{2(\sigma_S^+(A))^2}\norm{Ax-b}_2^2.
\end{equation}
Thus the local error-bound constant is \textcolor{black}{$\theta_{+,\mathrm{loc}}(\hat x)=2(\sigma_S^+(A))^2$}.
\end{proposition}

The injectivity condition is weaker than full column rank. It only excludes null-space directions that are feasible to first order at $\hat x$. Once the coordinates in $S$ are positive in the iterate, the Bregman distance becomes Euclidean on the local cone. The proof is given in Appendix~\ref{app:local-proofs}.

\begin{proposition}[Local entropy error bound on the simplex]
\label{prop:local-simplex}
Let
\[
    f(x)=\sum_{j=1}^n x_j\log x_j+\ind_{\Delta_n}(x),
    \qquad
    \nabla f^*(u)=\softmax(u).
\]
Assume $\hat x\in\operatorname{relint}(\Delta_n)$ and set \textcolor{black}{$\alpha=\min_j\hat x_j>0$}.
Define \textcolor{black}{$T_\Delta=\{d\in\RR^n:\one^\top d=0\}$ and $\sigma_\Delta(A)=\inf_{\substack{d\in T_\Delta\\\norm{d}_2=1}}\norm{Ad}_2$}.
Assume $\ker(A)\cap T_\Delta=\{0\}$, so $\sigma_\Delta(A)>0$. Let $\delta\in(0,1)$. If $x\in\Delta_n$ satisfies \textcolor{black}{\mbox{$\norm{x-\hat x}_\infty\le\delta\alpha$}}, then
\begin{equation}
\label{eq:local-simplex}
    \KL(\hat x\Vert x)
    \le
    \frac{1}{2(1-\delta)^2\alpha\sigma_\Delta(A)^2}
    \norm{Ax-b}_2^2.
\end{equation}
Hence one may use the local constant \textcolor{black}{$\theta_{\Delta,\mathrm{loc}}(\hat x)=2(1-\delta)^2\alpha\sigma_\Delta(A)^2$}.
\end{proposition}

A convenient equivalent condition for simplex injectivity is
\[
    \operatorname{rank}
    \begin{pmatrix}
        A\\
        \one^\top
    \end{pmatrix}=n.
\]
The proof of Proposition~\ref{prop:local-simplex} is given in Appendix~\ref{app:local-proofs}. If the neighborhood is forward invariant, Theorem~\ref{thm:selector-contraction} applies with $\theta_{\Delta,\mathrm{loc}}$. The proposition alone does not establish forward invariance.

\subsection{Summary of the theoretical hierarchy}
The results can now be read from general to specific.
\begin{enumerate}
    \item Proposition~\ref{prop:one-step} depends only on convexity and Bregman geometry.
    \item Theorems~\ref{thm:det-soft} and~\ref{thm:stoch-soft} add relative smoothness and an exact or unbiased oracle.
    \item Theorem~\ref{thm:ls-global} uses the exact least-squares second moment and allows a nonzero optimal residual.
    \item Theorem~\ref{thm:ls-linear} adds interpolation and a residual error bound to obtain last-iterate linear convergence.
    \item Theorem~\ref{thm:selector-contraction} concerns a different target: a strongly convex selector under hard equations.
    \item Theorem~\ref{thm:global-eb-sufficient} and Propositions~\ref{prop:local-nonnegative}--\ref{prop:local-simplex} explain when the selector error bound is available.
\end{enumerate}
This hierarchy is the reason for keeping two algorithmic branches. The master formulation is unified, while the strongest rates depend on the target regime and the geometry.

\section{Numerical experiments}
\label{sec:numerics}
\vl{The experiments complement the theory and show how the row-action updates behave in practice. They address three questions.}
\begin{enumerate}
    \item Does noisy constrained least squares show the stepsize trade-off predicted by Theorem~\ref{thm:ls-global}?
    \item Does the entropy selector satisfy the local residual-to-KL relation from Proposition~\ref{prop:local-simplex}?
    \item Can the soft entropic row step be used inside a structured matrix factorization scheme?
\end{enumerate}
\vl{The first two tests directly examine the theoretical predictions. The third shows how the soft entropic row step can be embedded in a simplex-structured NMF scheme and compares the resulting factors with an optimized deterministic reference.}

All figures and numerical results reported in this section can be
reproduced using the Python code available at
\url{https://github.com/vleplat/randomized-row-action-framework}.
The repository contains standalone experiment scripts, reusable
implementations of the row-action methods, and the corresponding
notebooks.

\begin{table}[t]
\centering
\caption{Summary of the numerical experiments.}
\label{tab:num-summary}
\small
\begin{tabularx}{\textwidth}{@{}L{0.23\textwidth}L{0.20\textwidth}L{0.22\textwidth}Y@{}}
\toprule
Experiment & Dimensions & Branch & Main quantity \\
\midrule
Noisy NNLS & $A\in\RR^{1000\times200}$ & Soft primal & weighted and last-iterate objective gaps \\
Entropy selector & $A\in\RR^{1000\times200}$ & Hard equality & KL distance, residual, and error-bound ratio \\
Simplex NMF & $X\in\RR_+^{159\times2500}$ & Soft block updates & reconstruction error and factor interpretation \\
\bottomrule
\end{tabularx}
\end{table}

\subsection{Noisy nonnegative least squares}
\label{subsec:noisy-nnls}
We consider
\begin{equation}
\label{eq:nnls-test}
    \min_{x\ge0} f(x):=\frac12\norm{Ax-b}_2^2.
\end{equation}
The matrix $A\in\RR^{1000\times200}$ has independent Gaussian entries scaled by $1/\sqrt{200}$. The nonnegative reference vector has $25\%$ nonzero entries, sampled uniformly from $[0.5,2]$. We set
\begin{equation}
\label{eq:nnls-noise}
    b=Ax_{\rm true}+e,
    \qquad
    \frac{\norm{e}_2}{\norm{Ax_{\rm true}}_2}=0.05.
\end{equation}
Thus the system is inconsistent and $f^\star>0$. A high-accuracy deterministic NNLS solver provides the reference solution and the value $f^\star$.

We apply \eqref{eq:projected-cls} with $C=\RR_+^{200}$. The row is sampled with probability proportional to its squared norm. One epoch contains $1000$ row updates. We compare four stepsize rules:
\[
    \eta_k=0.9,
    \qquad
    \eta_k=0.2,
    \qquad
    \eta_k=\frac{\eta_0}{\sqrt T},
    \qquad
    \eta_k=\frac{\eta_0}{\sqrt{k+1}},
\]
with $\eta_0=0.9$ for the two square-root rules. In the horizon rule, $T$ is the total number of row updates in the run, so the stepsize is constant inside that run.

Theorem~\ref{thm:ls-global} controls the weighted iterate
\begin{equation}
\label{eq:num-weighted-iterate}
    \bar x_T=
    \frac{\sum_{k=0}^{T-1}\eta_kx_k}
         {\sum_{k=0}^{T-1}\eta_k}.
\end{equation}
Figure~\ref{fig:nnls-theory} reports the corresponding objective gap and the horizon experiment. For the second panel, we use
\[
    T\in\{10^4,2.5\cdot10^4,5\cdot10^4,10^5,
    2.5\cdot10^5,5\cdot10^5,10^6,2\cdot10^6\}.
\]
These values correspond to $10$ through $2000$ epochs. The final weighted gap decreases as the horizon grows. The empirical curve is at least as steep as the displayed $T^{-1/2}$ reference over the tested range.

\begin{figure}[ht!]
\centering
\begin{subfigure}[t]{0.49\textwidth}
    \centering
    \includegraphics[width=\linewidth]{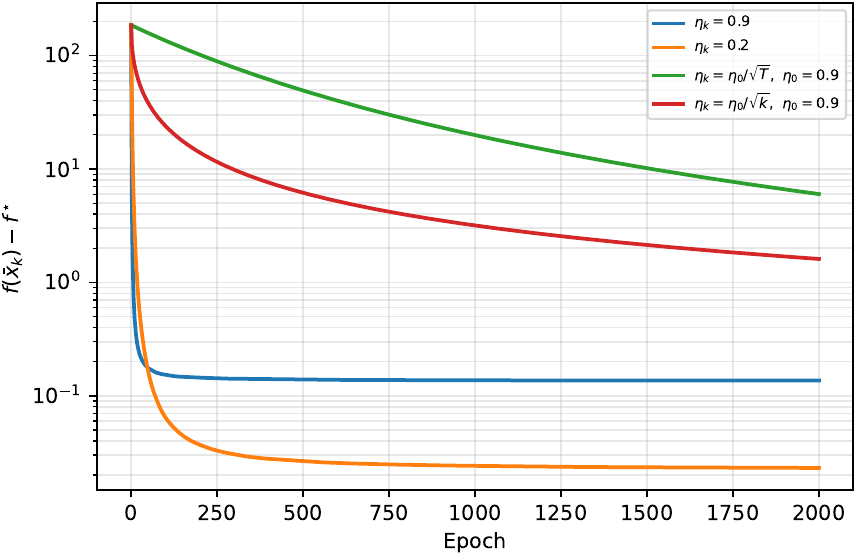}
    \caption{Weighted-average objective gap.}
\end{subfigure}
\hfill
\begin{subfigure}[t]{0.49\textwidth}
    \centering
    \includegraphics[width=\linewidth]{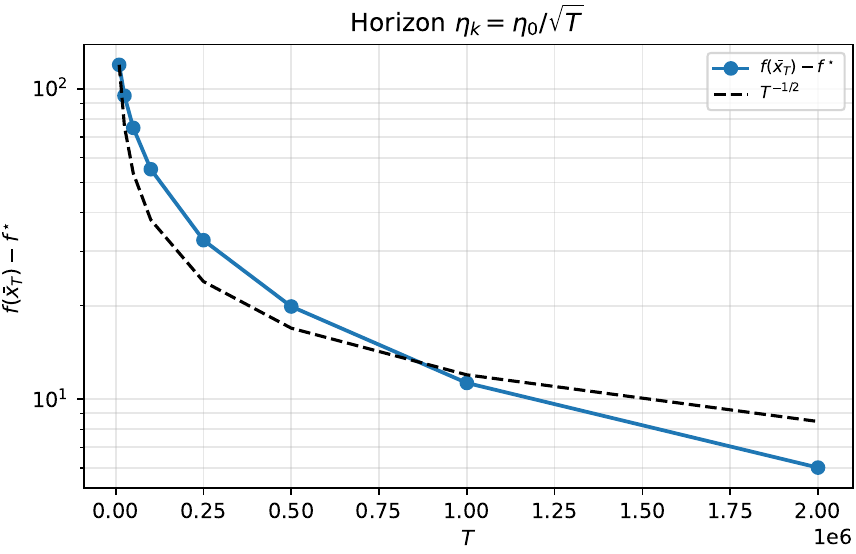}
    \caption{Final gap versus horizon.}
\end{subfigure}
\caption{Noisy NNLS. The left panel compares constant and decreasing stepsizes. The right panel uses $\eta_k=\eta_0/\sqrt T$ and compares the final weighted gap with a $T^{-1/2}$ reference.}
\label{fig:nnls-theory}
\end{figure}

The two constant stepsizes show the expected noisy-regime trade-off. The larger step decreases faster at the beginning but settles in a larger neighborhood. The smaller step is slower and reaches a smaller final neighborhood. The horizon rule reduces the final stochastic term when the computational budget grows, but it can be conservative from the first iteration of a long run.

Figure~\ref{fig:nnls-last} shows the last-iterate gap. This quantity is not controlled directly by Theorem~\ref{thm:ls-global}; it is reported as a practical diagnostic. The diminishing rule takes larger steps during the early part of the run than the horizon-tuned constant and then anneals. It therefore combines a faster transient with continued reduction of the stochastic error in this experiment.

\begin{figure}[ht!]
    \centering
    \includegraphics[width=0.72\textwidth]{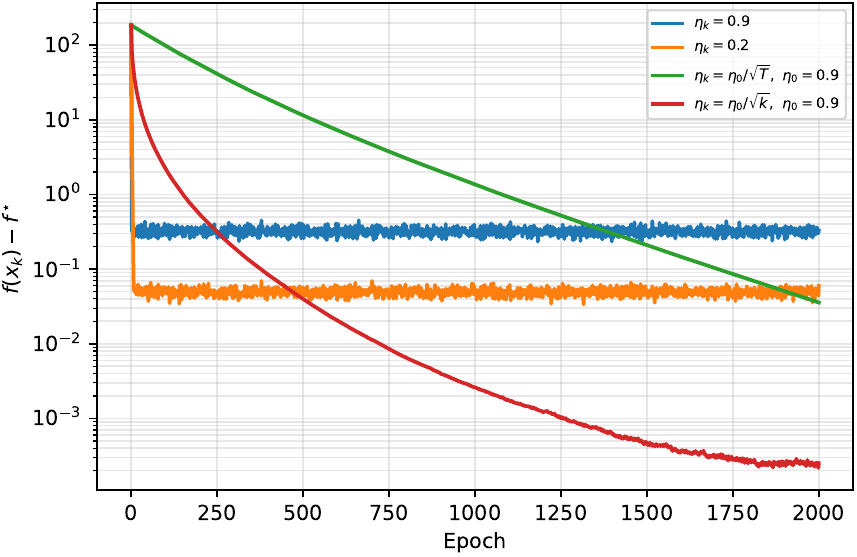}
    \caption{Noisy NNLS: last-iterate objective gap. Constant stepsizes reach a stepsize-dependent neighborhood, while the decreasing rules continue to reduce the error.}
    \label{fig:nnls-last}
\end{figure}

\vl{The last-iterate experiment complements the ergodic theory. With a constant stepsize, the method makes fast initial progress and then reaches a stepsize-dependent neighborhood. With a decreasing stepsize, the objective gap continues to decrease.}

\subsection{Entropy selector on the simplex}
\label{subsec:entropy-selector}
We next test Branch II with
\begin{equation}
\label{eq:entropy-test}
    \min_{x\in\Delta_n}\sum_{j=1}^n x_j\log x_j
    \qquad\text{subject to}\qquad
    Ax=b.
\end{equation}
We generate $A\in\RR^{1000\times200}$ with Gaussian entries scaled by $1/\sqrt{200}$. An interior point $\hat x$ is generated by applying a softmax map to a small random vector, and we set $b=A\hat x$. The matrix is injective on the simplex tangent space. It has full column rank in the reported test, so the feasible point is unique. 
\zy{This choice allows us to focus on the local geometry and contraction mechanism.}

The entropy mirror map is \textcolor{black}{\mbox{$x_k=\softmax(u_k)$}}.
\vl{We use Algorithm~\ref{alg:hard} with $\eta=1$ and choose the initial dual vector so that $x_0$ lies in the neighborhood required by Proposition~\ref{prop:local-simplex}. This setup allows the experiment to illustrate directly the local contraction mechanism described by the proposition.}

Figure~\ref{fig:selector-decay} reports the KL distance and the squared feasibility residual. Both quantities decrease steadily. Their nearly linear profiles on the semilogarithmic scale are consistent with a local linear regime.

\begin{figure}[ht!]
\centering
\begin{subfigure}[t]{0.49\textwidth}
    \centering
    \includegraphics[width=\linewidth]{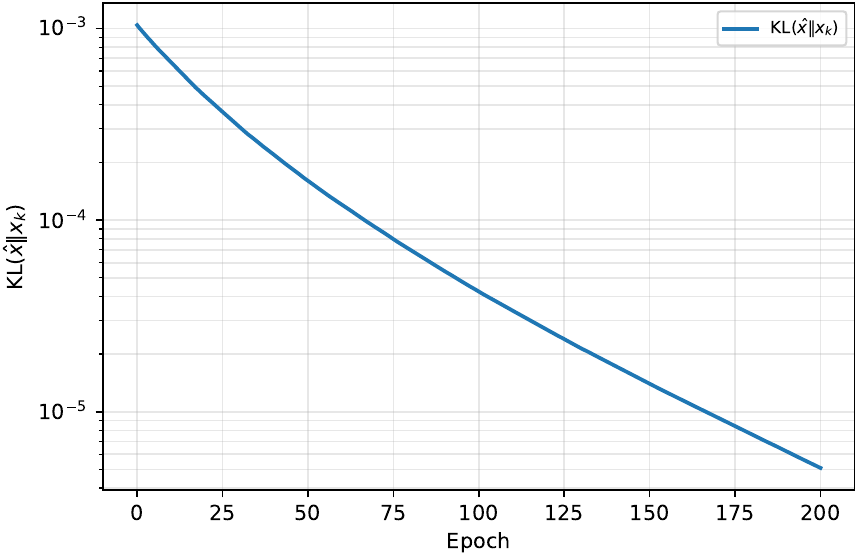}
    \caption{Entropy Bregman distance.}
\end{subfigure}
\hfill
\begin{subfigure}[t]{0.49\textwidth}
    \centering
    \includegraphics[width=\linewidth]{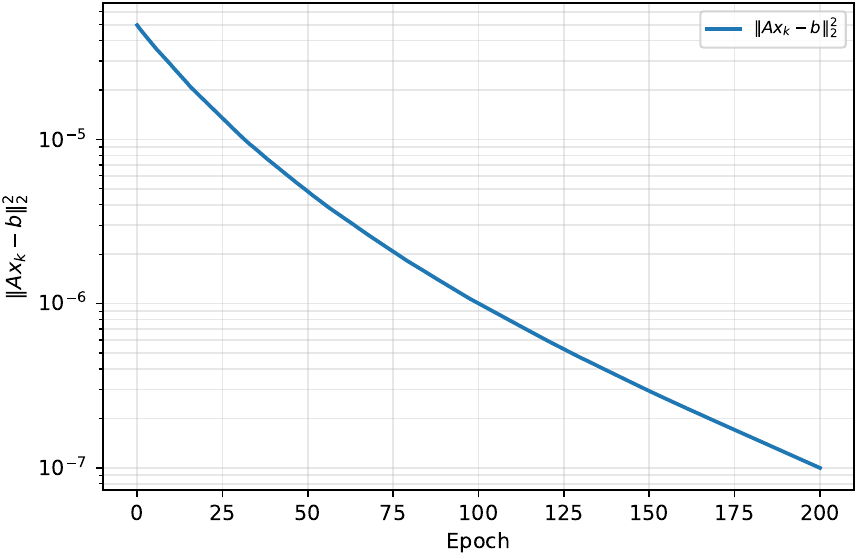}
    \caption{Squared feasibility residual.}
\end{subfigure}
\caption{Entropy Bregman selector on the simplex. The KL distance and the feasibility residual decrease in the local regime.}
\label{fig:selector-decay}
\end{figure}

Let \textcolor{black}{$\alpha=\min_j\hat x_j$ and $\delta=\tfrac{\norm{x_0-\hat x}_\infty}{\alpha}$}.
The initialization is chosen so that $\delta<1$. Proposition~\ref{prop:local-simplex} predicts
\begin{equation}
\label{eq:num-ratio-bound}
    \frac{\KL(\hat x\Vert x_k)}{\norm{Ax_k-b}_2^2}
    \le\frac{1}{\theta_\delta},
    \qquad
    \theta_\delta=2(1-\delta)^2\alpha\sigma_\Delta(A)^2,
\end{equation}
provided the iterates stay in the neighborhood. Figure~\ref{fig:selector-ratio} compares the empirical ratio with $1/\theta_\delta$ and with the limiting reference obtained by setting $\delta=0$. The empirical values remain below both displayed references.

\begin{figure}[ht!]
    \centering
    \includegraphics[width=0.72\textwidth]{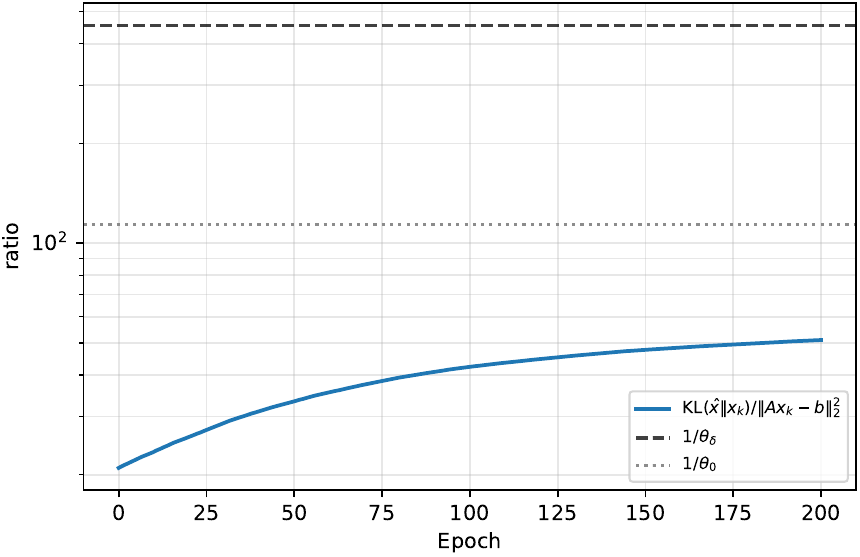}
    \caption{Entropy selector: the empirical ratio $\KL(\hat x\Vert x_k)/\norm{Ax_k-b}_2^2$ remains below the local theoretical reference levels.}
    \label{fig:selector-ratio}
\end{figure}

This experiment illustrates the full local mechanism used in Theorem~\ref{thm:selector-contraction}: the residual decays, the Bregman distance decays, and the residual controls the Bregman distance through a computable restricted singular value.

\subsection{Simplex-structured NMF}
\label{subsec:nmf}
We finally consider
\begin{equation}
\label{eq:simplex-nmf}
    \min_{W,H}\frac12\norm{X-WH}_F^2
    \qquad\text{subject to}\qquad
    W\ge0,
    \quad h_j\in\Delta_r.
\end{equation}
In applications such as hyperspectral unmixing, the columns of \(W\) represent
endmember signatures, while the columns of \(H\) represent abundance vectors~\cite{Gillis2020NMF}.
The simplex constraint encodes the fact that each data point is represented as
a convex combination of the endmembers.
The problem is nonconvex in $(W,H)$ and convex in each block separately. We use a two-block alternating scheme. The first two methods below update the blocks by randomized entropic row-action steps. We compare them with a deterministic simplex-constrained NMF reference.

\paragraph{Single-row method with fixed stepsizes}
With $W$ fixed, let $w_i^\top$ be row $i$ of $W$. For a sampled row $i$, define
\[
    \rho_i=w_i^\top H-X_{i,:},
    \qquad
    S_i^H=\frac{w_i\rho_i}{\norm{w_i}_2^2}.
\]
The abundance update is
\begin{equation}
\label{eq:H-update}
    H^+=\operatorname{NormalizeColumns}
    \bigl(H\odot\exp(-\eta_HS_i^H)\bigr).
\end{equation}
The column normalization is part of the entropy proximal map. It preserves $h_j\in\Delta_r$ without a Euclidean projection.

With $H$ fixed, let $h_j$ and $x_j$ denote column $j$ of $H$ and $X$. For a sampled column $j$, define
\[
    r_j=Wh_j-x_j,
    \qquad
    S_j^W=\frac{r_jh_j^\top}{\norm{h_j}_2^2}.
\]
The basis update is
\begin{equation}
\label{eq:W-update}
    W^+=W\odot\exp(-\eta_WS_j^W).
\end{equation}
A small positive floor is used only to avoid numerical underflow. In the experiment, we use the fixed values
\[
    \eta_H=0.1,
    \qquad
    \eta_W=0.01.
\]
One outer iteration contains $m$ sampled row updates for $H$ and $n$ sampled column updates for $W$.

\paragraph{Mini-batch row-action method}
The mini-batch variant replaces one sampled row or column by a small block. Let
$I\subset\{1,\ldots,m\}$ and $J\subset\{1,\ldots,n\}$ satisfy
$|I|=|J|=q$. Define
\[
    R_I^H=W_{I,:}H-X_{I,:},
    \qquad
    G_I^H=W_{I,:}^\top R_I^H,
    \qquad
    L_I^H=\norm{W_{I,:}}_2^2,
\]
and
\[
    R_J^W=WH_{:,J}-X_{:,J},
    \qquad
    G_J^W=R_J^W H_{:,J}^\top,
    \qquad
    L_J^W=\norm{H_{:,J}}_2^2.
\]
The corresponding updates are
\begin{align}
\label{eq:H-minibatch-update}
    H^+
    &=
    \operatorname{NormalizeColumns}
    \left(
        H\odot
        \exp\left(-\frac{\eta_H}{L_I^H}G_I^H\right)
    \right),\\
\label{eq:W-minibatch-update}
    W^+
    &=
    W\odot
    \exp\left(-\frac{\eta_W}{L_J^W}G_J^W\right).
\end{align}
The case $q=1$ recovers \eqref{eq:H-update}--\eqref{eq:W-update}. We use $q=8$, which gave the best speed--quality compromise in a preliminary ablation. One outer iteration processes $\lceil m/q\rceil$ abundance blocks and $\lceil n/q\rceil$ basis blocks, so that the amount of accessed data is nearly the same as for the single-row method. The derivation of \eqref{eq:H-minibatch-update}--\eqref{eq:W-minibatch-update} is given in Appendix~\ref{app:nmf-minibatch-derivation}.

\paragraph{Deterministic reference}
As a reference, we use the deterministic simplex-constrained NMF implementation provided in the Python companion repository to the book of Gillis~\cite{Gillis2020NMF}.\footnote{\url{https://github.com/vleplat/nmfbook}}
The method uses two-block coordinate descent. Each convex block subproblem is handled by a Nesterov-type accelerated projected-gradient routine, with projection onto the nonnegative orthant for $W$ and onto the product of simplices for $H$. We use the model enforcing $\one^\top h_j=1$, with zero volume penalty, so that the comparison concerns the same reconstruction objective and constraints as \eqref{eq:simplex-nmf}.

\vl{This experiment applies the proposed convex row-action block updates inside a nonconvex alternating NMF scheme. It tests their practical behavior and the quality of the recovered factors.}

\begin{remark}[Stepsize selection]
\label{rem:nmf-stepsizes}
\vl{We deliberately use fixed stepsizes in the two row-action methods. Recent work develops entropy-specific Polyak stepsizes for linear systems and proves convergence without imposing a very small constant step~\cite{malitsky2026entropicmirrordescentlinear}. Adapting these rules to randomized block updates would require accounting for the changing block objective and the unknown optimal block value. This is a natural direction for further work. Here, the fixed-step experiment shows that simple randomized entropic updates can recover meaningful factors close to those obtained by an optimized deterministic reference.}
\end{remark}

\paragraph{Moffett Field data}
We use a $50\times50$ region of the AVIRIS Moffett Field hyperspectral image. After removing unsuitable bands, the data matrix has size $159\times2500$. The scene contains three dominant materials, commonly interpreted as water, vegetation, and soil \cite{5256272,10.1162/neco_a_01679}. The water component is known
to be difficult to estimate because its spectral signature has significantly smaller norm
than the other materials. Figure~\ref{fig:moffett-data} shows the selected region.

\begin{figure}[ht!]
    \centering
    \includegraphics[width=0.82\linewidth]{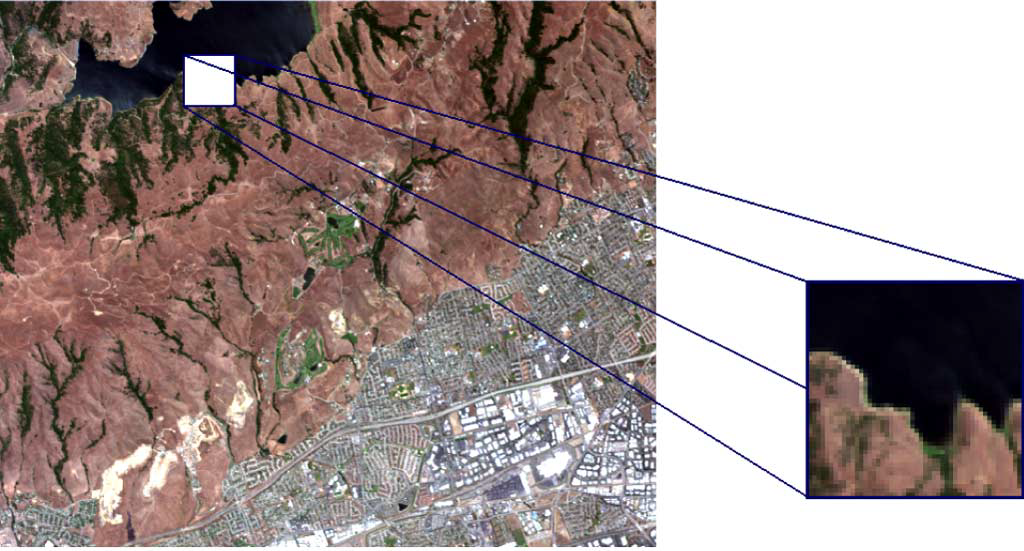}
    \caption{AVIRIS Moffett Field scene and the selected region of interest. The image is reproduced from Dobigeon et al.~\cite{5256272}.}
    \label{fig:moffett-data}
\end{figure}

\hsiupdate{We use rank $r=3$ and $200$ outer iterations for all three methods, and repeat the experiment with three random initializations. Within each run, all methods start from the same factors. Table~\ref{tab:moffett-comparison} reports a representative run. The mini-batch method reaches a smaller reconstruction error in less time than the single-row method. The deterministic reference gives the smallest reconstruction error and the shortest runtime.}

\begin{table}[ht!]
\centering

\caption{\hsiupdate{Moffett Field comparison for a representative initialization. All three methods use the same initial factors. Elapsed times are medians of three repeated runs with single-threaded BLAS; one-time JIT compilation is excluded.}}
\label{tab:moffett-comparison}
\begin{tabular}{lccc}
\toprule
Method & Objective & Relative error & \hsiupdate{Elapsed time (s)}\\
\midrule
Single row, fixed steps & $34.82$ & $4.13\times10^{-2}$ & \hsiupdate{$2.186$}\\
Mini-batch $q=8$ & $27.20$ & $3.65\times10^{-2}$ & \hsiupdate{$1.571$}\\
Deterministic reference & $21.58$ & $3.25\times10^{-2}$ & \hsiupdate{$1.352$}\\
\bottomrule
\end{tabular}
\end{table}

\hsiupdate{Across the three tested initializations, the median elapsed times are $2.170$~s for single-row, $1.570$~s for mini-batch $q=8$, and $1.333$~s for the deterministic reference. Thus, the mini-batch method takes approximately $18\%$ more time than the deterministic reference in this comparison.} Figure~\ref{fig:moffett-abundance-comparison} compares the abundance maps. The maps obtained by our two methods are close to the deterministic reference and recover the same main spatial structures. Figure~\ref{fig:moffett-spectral-comparison} gives the corresponding spectral signatures.

\begin{figure}[ht!]
    \centering
    \includegraphics[width=0.92\linewidth]{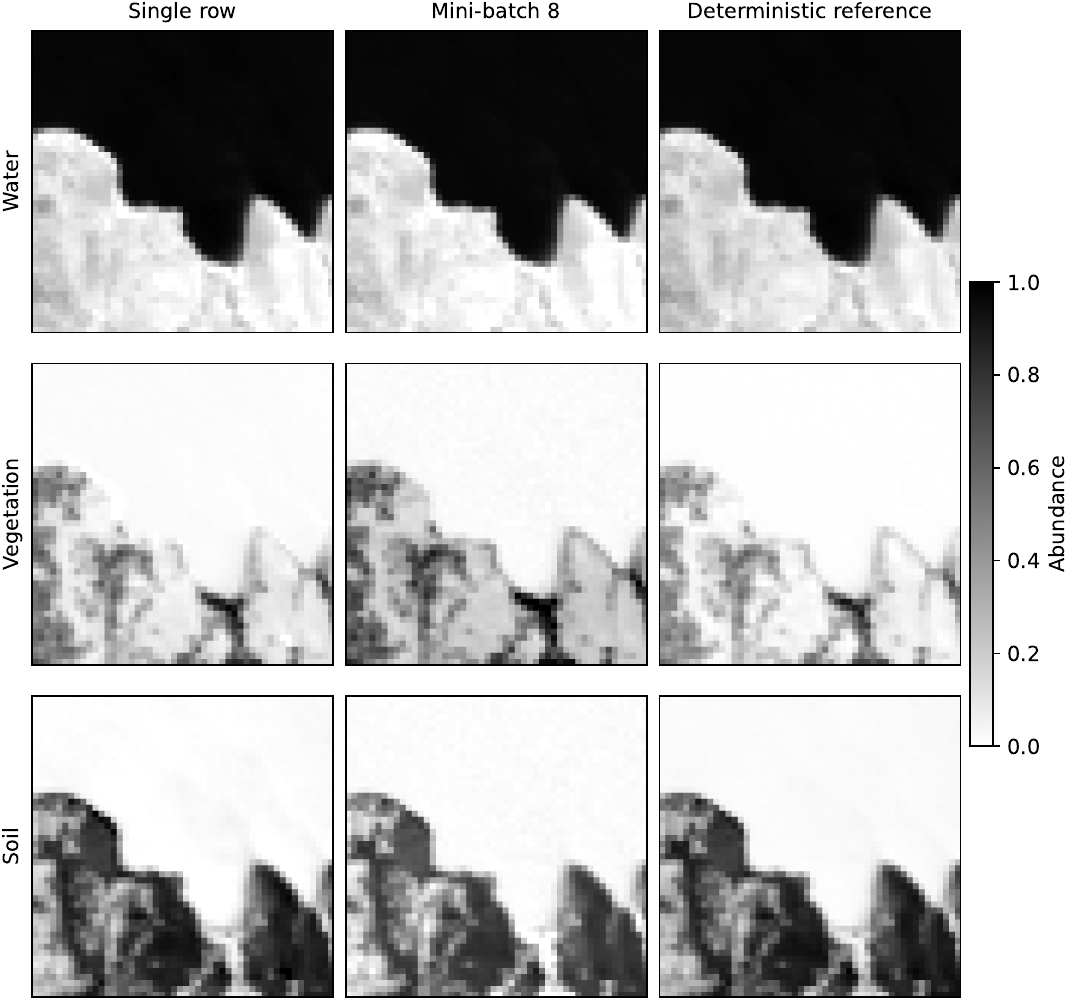}
    \caption{\vl{Moffett Field abundance maps. From left to right: fixed-step single-row method, mini-batch method with $q=8$, and deterministic reference. The components are aligned to the reference by spectral correlation.}}
    \label{fig:moffett-abundance-comparison}
\end{figure}

\begin{figure}[ht!]
    \centering
    \includegraphics[width=0.92\linewidth]{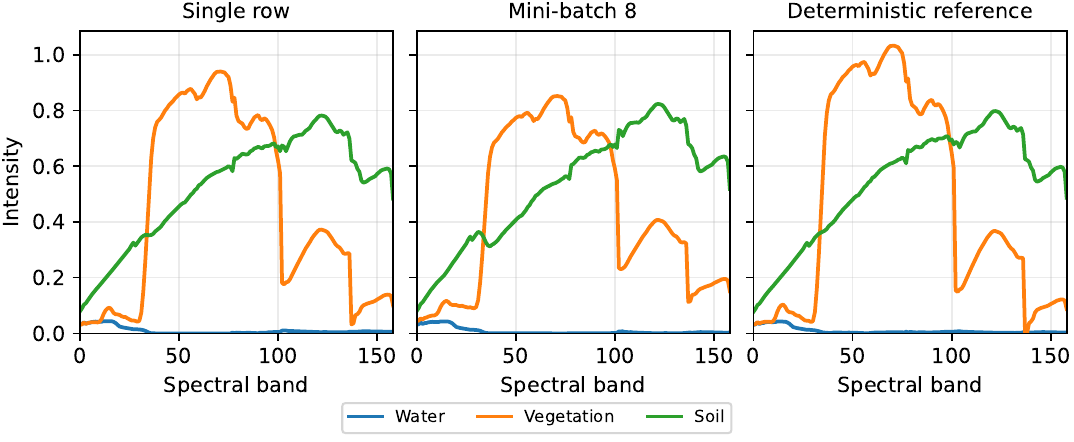}
    \caption{\vl{Estimated Moffett Field spectral signatures. From left to right: fixed-step single-row method, mini-batch method with $q=8$, and deterministic reference.}}
    \label{fig:moffett-spectral-comparison}
\end{figure}
\hsiupdate{Figure~\ref{fig:moffett-convergence-comparison} compares the relative reconstruction error against outer iterations and elapsed time. The mini-batch method gives better progress per unit of time than single-row, while the deterministic reference reaches the lowest reconstruction error in the shortest time.}

\begin{figure}[ht!]
    \centering
    \includegraphics[width=0.94\linewidth]{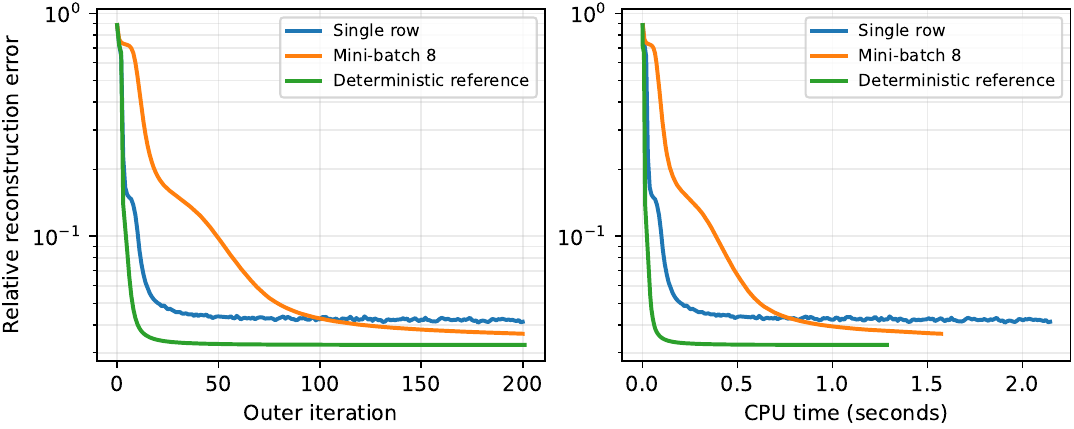}
    \caption{\hsiupdate{Relative reconstruction error for the representative initialization in Table~\ref{tab:moffett-comparison}, plotted against outer iterations and elapsed wall-clock time.}}
    \label{fig:moffett-convergence-comparison}
\end{figure}

Across the three tested initializations, the fixed single-row method was the most stable, while the mini-batch method achieved the best row-action result in two cases. This shows that mini-batching can substantially improve the practical performance, although the nonconvex NMF problem remains sensitive to initialization. 

Overall, both row-action variants recover high-quality and interpretable factors. In particular, the abundance maps and spectral signatures closely match those obtained by the optimized deterministic reference, while the best mini-batch results also achieve a comparable reconstruction error.

\section{Conclusion}
\label{sec:conclusion}

\vl{We introduced a unified model for randomized row-action methods with convex constraints and Bregman geometry. The framework covers soft constrained data fitting and hard equality-constrained selectors. In important special cases, it recovers established projected, classical, sparse, and entropic Kaczmarz-type methods. This common view separates the shared row-action structure from the assumptions needed for the two convergence analyses. The soft branch uses primal Bregman-proximal row steps, while the hard branch uses dual Bregman--Kaczmarz updates and a mirror map. Under relative smoothness, we obtain deterministic and stochastic ergodic bounds. For constrained least squares, the row structure gives sharper noisy guarantees and, under interpolation and a residual error bound, linear convergence of the projected iterates. For equality-constrained selectors, a residual-to-Bregman bound gives expected linear contraction. We also provide global and local conditions under which these estimates hold.}

\vl{The current framework is designed for row- or block-decomposable data terms with tractable local Bregman-proximal maps. In this setting, the selector analysis covers hard equality-constrained models and gives explicit local results for the nonnegative and simplex geometries. The simplex-NMF experiment also shows how the same row-action updates can be used inside a nonconvex alternating factorization scheme.}

A natural direction is to use these methods as inner solvers in block-coordinate algorithms for matrix and tensor approximation. The same ideas may also lead to stochastic or streaming methods for nonlinear factorizations, including ReLU-type models such as $\max(WH,0)$ \cite{Saul2022Nonlinear,GillisPorcelliSeraghiti2026}. The present framework provides a starting point for studying the stability and convergence of such extensions.

\appendix
\section{Proofs for the general soft branch}
\label{app:general-soft-proofs}
We give the proofs of Theorems~\ref{thm:det-soft} and~\ref{thm:stoch-soft}. They are placed here to keep the main theorem flow focused on the row-action specializations.

\subsection{Proof of Theorem~\ref{thm:det-soft}}
By relative smoothness,
\[
    \Phi(x_{k+1})
    \le
    \Phi(x_k)+\ip{\nabla\Phi(x_k)}{x_{k+1}-x_k}
    +L_\Phi D_h(x_{k+1},x_k).
\]
Convexity also gives
\[
    \Phi(x_k)-\Phi(x)
    \le\ip{\nabla\Phi(x_k)}{x_k-x}.
\]
Combining these two inequalities yields
\begin{equation}
\label{eq:app-det-combine}
    \Phi(x_{k+1})-\Phi(x)
    \le
    \ip{\nabla\Phi(x_k)}{x_{k+1}-x}
    +L_\Phi D_h(x_{k+1},x_k).
\end{equation}
Multiply \eqref{eq:app-det-combine} by $\eta_k$ and combine it with \eqref{eq:exact-one-step}, using $s_k=\nabla\Phi(x_k)$. We obtain
\begin{align*}
    \eta_k(F(x_{k+1})-F(x))
    &\le
    D_h(x,x_k)-D_h(x,x_{k+1})
    -(1-\eta_kL_\Phi)D_h(x_{k+1},x_k).
\end{align*}
Since $\eta_k\le1/L_\Phi$, the last term is nonpositive and may be dropped. Summing from $k=0$ to $K-1$ gives \eqref{eq:det-soft-sum}. The Bregman terms telescope. Jensen's inequality then gives
\[
    F(\bar x_K)
    \le\frac{1}{S_K}\sum_{k=0}^{K-1}\eta_kF(x_{k+1}),
\]
which proves \eqref{eq:det-soft-rate}.

\subsection{Proof of Theorem~\ref{thm:stoch-soft}}
Start again from \eqref{eq:exact-one-step}, now with
\[
    s_k=\nabla\Phi(x_k)+\delta_k.
\]
Using the same relative-smoothness estimate as in \eqref{eq:app-det-combine},
\begin{align}
\label{eq:app-stoch-one-step}
    \eta_k(F(x_{k+1})-F(x))
    &\le
    D_h(x,x_k)-D_h(x,x_{k+1})
    -(1-\eta_kL_\Phi)D_h(x_{k+1},x_k)\\
    &\quad-\eta_k\ip{\delta_k}{x_{k+1}-x}.\notag
\end{align}
Decompose \textcolor{black}{$\ip{\delta_k}{x_{k+1}-x}=\ip{\delta_k}{x_{k+1}-x_k}+\ip{\delta_k}{x_k-x}$}.
The second term has zero conditional expectation. Strong convexity gives
\[
    D_h(x_{k+1},x_k)
    \ge\frac{\mu_h}{2}\norm{x_{k+1}-x_k}^2.
\]
Hence Young's inequality yields
\[
    \textcolor{black}{\eta_k\norm{\delta_k}_*\norm{x_{k+1}-x_k}
    -(1-\eta_kL_\Phi)D_h(x_{k+1},x_k)
    \le\frac{\eta_k^2}{2\mu_h(1-\eta_kL_\Phi)}\norm{\delta_k}_*^2.}
\]
Take conditional expectation in \eqref{eq:app-stoch-one-step}, use Assumption~\ref{ass:bounded-variance}, and then take total expectation. Since $\eta_k\le\bar\eta$,
\[
    \textcolor{black}{\eta_k\EE[F(x_{k+1})-F(x)]
    \le\EE[D_h(x,x_k)-D_h(x,x_{k+1})]
    +\frac{\eta_k^2\sigma^2}{2\mu_h(1-\bar\eta L_\Phi)}.}
\]
Summing proves \eqref{eq:stoch-soft-sum}. Jensen's inequality gives \eqref{eq:stoch-soft-rate}.

\section{Proofs of the local selector bounds}
\label{app:local-proofs}

\subsection{Proof of Proposition~\ref{prop:local-nonnegative}}
For \textcolor{black}{$f(x)=\ind_{\RR_+^n}(x)+\tfrac12\norm{x}_2^2$}, we have \textcolor{black}{\mbox{$\partial f(x)=x+N_{\RR_+^n}(x)$}}.
Since $x_S>0$, the normal-cone term vanishes on $S$, so $u_S=x_S$. Since $\hat x_{S^c}=0$, a coordinate-wise calculation gives \textcolor{black}{\mbox{$D_f^u(x,\hat x)=\tfrac12\norm{x-\hat x}_2^2$}}.
Indeed, on a coordinate $j\in S$ this is the ordinary quadratic Bregman distance. On $S^c$, either $x_j>0$ and $u_j=x_j$, or $x_j=0$ and the contribution is zero.

Moreover, $x-\hat x\in\cT_S$. Therefore
\[
    \norm{Ax-b}_2^2
    =\norm{A(x-\hat x)}_2^2
    \ge(\sigma_S^+(A))^2\norm{x-\hat x}_2^2.
\]
Combining the two identities proves \eqref{eq:local-nonnegative}.

\subsection{Proof of Proposition~\ref{prop:local-simplex}}
Set \textcolor{black}{$t_j=\tfrac{x_j}{\hat x_j}$}.
The neighborhood assumption gives $t_j\ge1-\delta>0$. Taylor's formula for $-\log t$ around $1$ gives
\[
    -\log t_j
    \le-(t_j-1)+\frac{(t_j-1)^2}{2(1-\delta)^2}.
\]
Multiply by $\hat x_j$ and sum. The linear term vanishes because $x$ and $\hat x$ both belong to the simplex. Hence
\begin{align*}
    \KL(\hat x\Vert x)
    &\le
    \frac{1}{2(1-\delta)^2}
    \sum_j\frac{(x_j-\hat x_j)^2}{\hat x_j}\\
    &\le
    \frac{1}{2(1-\delta)^2\alpha}\norm{x-\hat x}_2^2.
\end{align*}
Since $x-\hat x\in T_\Delta$ and $A\hat x=b$,
\[
    \norm{Ax-b}_2^2
    =\norm{A(x-\hat x)}_2^2
    \ge\sigma_\Delta(A)^2\norm{x-\hat x}_2^2.
\]
Combining the two estimates proves \eqref{eq:local-simplex}.

\section{Exact one-row Bregman projections}
\label{app:exact-bregman-projection}
The relaxed selector update is closely related to the exact Bregman projection onto one row hyperplane. This appendix makes the relation explicit.

Let \textcolor{black}{$H_i=\{x\in\RR^n:a_i^\top x=b_i\}$}, and let $u_k\in\partial f(x_k)$. The Bregman projection of $x_k$ onto $H_i$ is
\begin{equation}
\label{eq:exact-bregman-projection}
    x_{k+1}\in\argmin_{x\in H_i}D_f^{u_k}(x_k,x).
\end{equation}
The terms independent of $x$ may be removed. The optimality condition gives a scalar multiplier $t_k$ such that
\[
    u_{k+1}=u_k-t_ka_i,
    \qquad
    x_{k+1}=\nabla f^*(u_{k+1}),
\]
and $t_k$ solves
\begin{equation}
\label{eq:scalar-bregman-root}
    a_i^\top\nabla f^*(u_k-t_ka_i)=b_i.
\end{equation}
Thus an exact row projection requires only a one-dimensional root problem when the mirror map is available.

In Euclidean geometry, $\nabla f^*(u)=u$ and \eqref{eq:scalar-bregman-root} has the closed form \textcolor{black}{\mbox{$t_k=\tfrac{a_i^\top x_k-b_i}{\norm{a_i}_2^2}$}}.
The exact Bregman projection then coincides with the classical Kaczmarz step. For a general geometry, Algorithm~\ref{alg:hard} replaces the exact scalar $t_k$ by \textcolor{black}{$\eta_k\tfrac{a_i^\top x_k-b_i}{\norm{a_i}_2^2}$}.
This avoids a root solve and gives the relaxed contraction in Theorem~\ref{thm:selector-contraction}.

This derivation also explains the overlap between the two branches. In the Euclidean consistent case, the primal projected row step and the dual selector step are the same update. Outside this case, they have different meanings. In a soft inconsistent problem, the row hyperplane is not a constraint of the target problem. In a non-Euclidean selector, the dual variable is needed to preserve the selected Bregman geometry.

\section{Derivation of the simplex-NMF row and mini-batch updates}
\label{app:nmf-minibatch-derivation}

We derive the updates used in Section~\ref{subsec:nmf}. The derivation is the same for a single row or column and for a sampled block. The only difference is the local gradient and its normalization.

\subsection{Abundance block}

Fix $W$ and let $I\subset\{1,\ldots,m\}$ be a sampled set of spectral rows. Define the local quadratic loss \textcolor{black}{$\phi_I(H)=\tfrac12\norm{W_{I,:}H-X_{I,:}}_F^2$}.
Its gradient is \textcolor{black}{\mbox{$\nabla\phi_I(H)=W_{I,:}^\top(W_{I,:}H-X_{I,:})=G_I^H$}}.
The Euclidean gradient of $\phi_I$ is Lipschitz with constant \textcolor{black}{\mbox{$L_I^H=\norm{W_{I,:}}_2^2$}}.
On the product of simplices, use the negative entropy
\textcolor{black}{$\mathcal H_\Delta(H)=\sum_{j=1}^n\sum_{\ell=1}^r H_{\ell j}\log H_{\ell j}$}.
For a current positive iterate $H$, the normalized entropic proximal step is
\begin{equation}
\label{eq:app-H-prox}
    H^+
    \in
    \argmin_{\substack{Z\ge0\\ \one^\top Z=\one^\top}}
    \left\{
        \frac{\eta_H}{L_I^H}\ip{G_I^H}{Z}
        +
        D_{\mathcal H_\Delta}(Z,H)
    \right\}.
\end{equation}
The Frobenius inner product is used in \eqref{eq:app-H-prox}. Since the columns are independent, consider one column $z_j$. Introducing a multiplier $\lambda_j$ for $\one^\top z_j=1$, the optimality condition is \textcolor{black}{$\log\tfrac{(z_j)_\ell}{(h_j)_\ell}+\tfrac{\eta_H}{L_I^H}(G_I^H)_{\ell j}+\lambda_j=0$}.
Hence \textcolor{black}{$(z_j)_\ell=(h_j)_\ell\exp\left(-\tfrac{\eta_H}{L_I^H}(G_I^H)_{\ell j}\right)e^{-\lambda_j}$}.
The multiplier is determined by $\one^\top z_j=1$. Therefore
\begin{equation}
\label{eq:app-H-closed-form}
    H^+
    =
    \operatorname{NormalizeColumns}
    \left(
        H\odot
        \exp\left(
            -\frac{\eta_H}{L_I^H}G_I^H
        \right)
    \right).
\end{equation}

For a singleton $I=\{i\}$, \textcolor{black}{we have $G_{\{i\}}^H=w_i(w_i^\top H-X_{i,:})$ and $L_{\{i\}}^H=\norm{w_i}_2^2$}.
Thus \eqref{eq:app-H-closed-form} reduces exactly to the single-row update \eqref{eq:H-update}.

\subsection{Basis block}

Fix $H$ and let $J\subset\{1,\ldots,n\}$ be a sampled set of columns. Define \textcolor{black}{$\psi_J(W)=\tfrac12\norm{WH_{:,J}-X_{:,J}}_F^2$}.

Its gradient is \textcolor{black}{\mbox{$\nabla\psi_J(W)=(WH_{:,J}-X_{:,J})H_{:,J}^\top=G_J^W$}},
with Euclidean Lipschitz constant \textcolor{black}{\mbox{$L_J^W=\norm{H_{:,J}}_2^2$}}.

On the positive orthant, use \textcolor{black}{$\mathcal H_+(W)=\sum_{a=1}^m\sum_{\ell=1}^r W_{a\ell}(\log W_{a\ell}-1)$}.
The entropic proximal step is
\begin{equation}
\label{eq:app-W-prox}
    W^+
    \in
    \argmin_{Z\ge0}
    \left\{
        \frac{\eta_W}{L_J^W}\ip{G_J^W}{Z}
        +
        D_{\mathcal H_+}(Z,W)
    \right\}.
\end{equation}
The entrywise optimality condition is \textcolor{black}{$\log\tfrac{Z_{a\ell}}{W_{a\ell}}+\tfrac{\eta_W}{L_J^W}(G_J^W)_{a\ell}=0$}, which gives
\begin{equation}
\label{eq:app-W-closed-form}
    W^+
    =
    W\odot
    \exp\left(
        -\frac{\eta_W}{L_J^W}G_J^W
    \right).
\end{equation}
In the implementation, entries smaller than a fixed machine-level threshold are replaced by that threshold to avoid underflow to exact zero.

For a singleton $J=\{j\}$, \textcolor{black}{we have $G_{\{j\}}^W=(Wh_j-x_j)h_j^\top$ and $L_{\{j\}}^W=\norm{h_j}_2^2$}.
Hence \eqref{eq:app-W-closed-form} reduces to the single-column update \eqref{eq:W-update}.

\subsection{Single-row and mini-batch sweeps}

For batch size $q$, one abundance sweep performs \textcolor{black}{$\left\lceil\tfrac{m}{q}\right\rceil$} independent block updates, and one basis sweep performs \textcolor{black}{$\left\lceil\tfrac{n}{q}\right\rceil$} independent block updates. Sampling is without replacement inside each sampled block, while different blocks are drawn independently. For $q=1$, the block update formulas reduce to the single-row and single-column updates, although the sampling schedules need not coincide. With $q=8$, nearly the same total amount of data is accessed through fewer and denser matrix operations.

\vl{The normalization by $L_I^H$ and $L_J^W$ gives a blockwise scale that is unchanged by a common rescaling of the sampled block. The resulting formulas provide practical realizations of the primal entropic update for the two convex block subproblems. Since the complete alternating NMF problem is nonconvex, its convergence is not a direct consequence of the convex analysis in Section~\ref{sec:theory}.}

\section{Additional remarks on implementation}
\label{app:implementation}
This appendix records details that are useful for reproducing the experiments but are not needed for the theory.

\paragraph{Weighted averages}
For the NNLS experiment, the weighted iterate is accumulated online:
\[
    z_{k+1}=z_k+\eta_kx_k,
    \qquad
    s_{k+1}=s_k+\eta_k,
    \qquad
    \bar x_{k+1}=z_{k+1}/s_{k+1}.
\]
No iterates need to be stored. The objective is evaluated once per epoch.

\paragraph{Restricted singular value on the simplex}
Let $Q\in\RR^{n\times(n-1)}$ have orthonormal columns spanning $T_\Delta$. Then \textcolor{black}{\mbox{$\sigma_\Delta(A)=\sigma_{\min}(AQ)$}}.
In the experiment, $Q$ is obtained from a QR factorization of a basis of $\one^\perp$. The same construction is used to evaluate the two reference levels in Figure~\ref{fig:selector-ratio}.

\paragraph{Stable entropy updates}
For normalized softmax evaluations, subtracting the largest exponent leaves the result unchanged. In the HSI experiment, exponential arguments in the multiplicative updates are clipped to a bounded interval before exponentiation. The columns of $H$ are normalized to the simplex, and a small positive floor is applied to $W$ to avoid underflow to exact zero.

\paragraph{\hsiupdate{Implementation and timing of the HSI experiment}}
\hsiupdate{The single-row method uses compiled update kernels. For mini-batch $q=8$, the block spectral-norm constants are computed from eigenvalues of small Gram matrices, with independent eigenproblems evaluated in batches. The reported elapsed times include the solver and its diagnostic evaluations, but exclude the one-time compilation of the single-row kernels.}

\paragraph{Stopping and monitoring}
The synthetic tests use a prescribed number of epochs because the purpose is to show the dependence on the horizon and the local contraction. The NMF experiment monitors the relative error \textcolor{black}{$\tfrac{\norm{X-WH}_F}{\norm{X}_F}$} after every outer iteration. The factors are initialized with positive random entries; the columns of $H$ are normalized to the simplex before the first update.



\bibliographystyle{siamplain}
\bibliography{row_action_refs_simax}

\end{document}